\documentclass[final]{siamltex}
\usepackage{amsmath,bm}
\usepackage{algorithm2e}
\usepackage{amsfonts}
\usepackage{graphicx}
\DeclareMathAlphabet\mathpzc{OT1}{pzc}{m}{it}
\let\mathcal=\mathpzc

\let \la = \leftarrow

\newcommand\partialderiv[3][]{\frac{\partial^{#1}#2}{\partial{#3}^{#1}}}

\let\trueiiint=\iiint
\def\iiint{\mathop{\textstyle\trueiiint}\limits}
\def\intinfty{\int\limits_{\!\!-\infty\,\,}^{\,\,\infty\!\!}\kern-0.0em}
\def\iintinfty{\mathop{\int\!\!\int}\limits_{\!\!-\infty\,\,}^{\,\,\infty\!\!}\kern-0.0em}
\def\iiintinfty{\mathop{\int\!\!\int\!\!\int}\limits_{\!\!-\infty\,\,}^{\,\,\infty\!\!}\kern-0.0em}

\let\<=\langle
\let\>=\rangle
\let\^=\hat
\def\~#1{{\-ox{\sf#1}}}
\def\N{{\mathbb N}}
\def\R{{\mathbb R}}
\let\-=\mathbf

\def\circ{\ifmmode\mathchar"220E\else$\mathchar"220E$\fi}
\def\@#1{{\cal #1}}

\title{An adaptive independence sampler MCMC algorithm for infinite-dimensional Bayesian inferences\thanks{The work was  supported by the NSFC, under grant 11301337.}}

\author{Zhe Feng\footnotemark[2]
        \and Jinglai Li\footnotemark[3]}

\begin{document}

\maketitle

\renewcommand{\thefootnote}{\fnsymbol{footnote}}

\footnotetext[2]{Department of Mathematics, Zhiyuan College,   
Shanghai Jiao Tong University, 800 Dongchuan Rd, Shanghai 200240, China, (sjtufz@sjtu.edu.cn).}
\footnotetext[3]{Corresponding Author, Institute of Natural Sciences, Department of Mathematics, and 
the MOE Key Laboratory of Scientific and Engineering Computing, 
Shanghai Jiao Tong University, 800 Dongchuan Rd, Shanghai 200240, China, (jinglaili@sjtu.edu.cn).}

\renewcommand{\thefootnote}{\arabic{footnote}}

\begin{abstract}
Many scientific and engineering problems require to perform Bayesian inferences in function spaces, in which the unknowns are of infinite dimension. 
In such problems, many standard  Markov Chain Monte Carlo (MCMC) algorithms  become arbitrary slow under the mesh refinement, which is referred to as being dimension dependent.  
In this work we develop an independence sampler based MCMC method for the infinite dimensional Bayesian inferences.  
We represent the proposal distribution as a mixture of a finite number of specially parametrized Gaussian measures.  
We show that under the chosen parametrization, the resulting MCMC algorithm is dimension independent. 
We also design an efficient adaptive algorithm to adjust the parameter values of the mixtures from the previous samples. 
Finally we provide numerical examples to demonstrate the efficiency and robustness of the proposed method, even for problems
with multimodal posterior distributions.  
\end{abstract}

\pagestyle{myheadings}
\thispagestyle{plain}
\markboth{Z. FENG AND J. LI}{ADAPTIVE MCMC FOR INFINITE DIMENSIONAL  INFERENCE}

\section{Introduction}\label{s:intro}
Nonparametric Bayesian inferences have applications in many scientific problems, ranging from regression~\cite{hjort2010bayesian} 
to inverse problems~\cite{kaipio2005statistical,stuart2010inverse}.
In those problems the unknown that we want to infer is of infinite-dimension, for example, a function of space or time. 
In many practical problems, the posterior distributions do not admit a closed form and need to be computed numerically.
Specifically one first represents the unknown function with a finite-dimensional parametrization, for example, by discretizing the function on a pre-determined mesh grid, and then solve the resulting finite dimensional inference problem with
the  Markov Chain Monte Carlo (MCMC) simulations. 
It has been known that standard MCMC algorithms, such as the random walk Metropolis-Hastings (RWMH),  can become arbitrarily slow as the discretization mesh of the unknown is refined~\cite{roberts1997weak,roberts2001optimal,beskos2009optimal,mattingly2012diffusion}. 
That is, the mixing time of an algorithm can increase to infinity as the dimension of the discretized parameter approaches to infinity,
and in this case the algorithm is said to be \emph{dimension-dependent}.
To this end, a very interesting line of research is to develop \emph{dimension-independent} MCMC algorithms by requiring the algorithms to be 
well-defined in the function spaces.
In particular, a family of dimension-independent MCMC algorithms were presented in \cite{cotter2013mcmc} by constructing a preconditioned Crank-Nicolson (pCN)
discretization of a stochastic partial differential equation (SPDE) that preserves the reference measure. 

Just like its finite dimensional counterparts, the sampling efficiency of the infinite dimensional MCMC can be improved by incorporating 
the data information in the proposal design.   
One way of doing so is to guide the proposal with the local derivative information of the likelihood function.
Methods in this category include: the stochastic Newton MCMC~\cite{martin2012stochastic,petra2014computational},
the operator-weighted proposal method~\cite{law2014proposals}, the infinite-dimensional Metropolis-adjusted Langevin
algorithm~(MALA)~\cite{beskos2008mcmc,Beskos201446},
and the dimension-independent likelihood-informed~(DILI) MCMC~\cite{cui2014dimension},  just to name a few.  
An alternative type of methods to improve the efficiency with the data information is the adaptive MCMC (c.f. ~\cite{andrieu2008tutorial,atchade2009adaptive,roberts2009examples} and 
the references therein), which automatically adjusts the proposal as the algorithm proceeds. While the first type of approaches utilize the gradient or the Hessian of the likelihood function
to accelerate the computation, the adaptive methods do not require such information, which makes it particularly convenient for problems with black-box models.

In this paper we propose an adaptive MCMC algorithm with independence sampler (IS)~\cite{tierney1994markov} for infinite dimensional problems. 
IS, also known as the independent Metropolis-Hastings~(MH)~\cite{holden2009adaptive}, or the Metropolized independent sampling~\cite{liu1996metropolized},
 is an alternative to the popular RWMH algorithm, which proposes from a stationary distribution, i.e., one that is independent of the present position. 
The design principle for the independence sampler method is rather straightforward: loosely speaking, one should choose the proposal distribution to be as close to the target distribution as possible.
The basic idea here is to represent the proposal distribution with a mixture of a finite number of parametrized Gaussian measures,
and optimize the parameters as the algorithm proceeds. Our specific parametrization ensures the algorithm to be well defined in function spaces and therefore dimension independent. 
As is mentioned earlier, a major advantage of the proposed method is that it can propose efficiently without using the derivative information
of the likelihood function. 
Moreover as is demonstrated by our numerical examples in Section~5, our method performs well for \emph{multimodal} posterior distributions which can be challenging for many existing algorithms.

The rest of the paper is organized as the following. In section~\ref{s:setup} we introduce the basic setup of the infinite dimensional Bayesian inference problem. 
In section~\ref{s:mixture} we present the Gaussian mixture based independence sampler
and show that it is well-defined in the function space. Section~4 is devoted to a detailed description of the complete algorithm  and
 and section~5 provides several numerical examples of the proposed method.  
\section{Problem setup}\label{s:setup}
We consider a separable {Hilbert} space $X$ with inner product $\<\cdot,\cdot\>_X$.
 Our goal is to estimate the unknown  $u\in X$ from data $y\in Y$ where $Y$ is the data space and $y$ is related to $u$ via the likelihood function 
\[L(u,y) = \frac1Z\exp(-\Phi^y(u)),\]
where $Z$ is a normalization constant.
In what follows, without causing any ambiguity, we shall drop the superscript $y$ in $\Phi^y$ for simplicity. 
In this work we require the functional $\Phi$ satisfies the Assumptions (6.1) in \cite{cotter2013mcmc}, i.e.,
\begin{description}
\item{(a)} there exists $q>0$, $Q>0$ such that, for all $u\in X$,
\[0\leq \Phi(u)\leq Q(1+\|u\|_X^q);\]
\item{(b)} for every $r>0$ there is $Q_r>0$ such that, for all $u,\,v\in X$ with \\
$\max\{\|u\|_X,\|v\|_X\}<r$,
\[|\Phi(u)-\Phi(v)|\leq Q_r \|u-v\|_X.\]
\end{description}
We do not have any restrictions on the space $Y$.

In the Bayesian inference we assume that the prior $\mu_0$ of $u$, is a  (without loss of generality)~zero-mean Gaussian measure defined on $X$ with covariance operator $C_0$,
i.e. $\mu_0 = N(0,C_0)$. 
Note that $C_0$ is symmetric positive and of trace class.
The range of $C_0^{\frac12}$,
\[E = \{u = C_0^{\frac12} x\, |\, x\in X\}\subset X,\]
which is a Hilbert space equipped with inner product~\cite{da2006introduction},
\[\<\cdot,\cdot\>_E = \<C_0^{-\frac12}\cdot,C_0^{-\frac12}\cdot\>_X ,\]
is called the Cameron-Martin space of measure $\mu_0$. 
In this setting, the posterior measure $\mu^y$ of $u$ conditional on data $y$
is provided by the Radon-Nikodym derivative:
\begin{equation} \frac{d\mu^y}{d\mu_{0}}(u) =\frac1Z\exp(-\Phi(u)), \label{e:bayes}
\end{equation}
which can be interpreted as the Bayes' rule in the infinite dimensional setting.
Our goal is to draw samples from the posterior $\mu^y$ with MCMC algorithms. 

Note that the definition of the maximum a posteriori (MAP) estimator in finite dimensional spaces does not apply here, as the measures $\mu^y$ and $\mu_0$ are not absolutely
continuously with respect to the Lebesgue measure; instead,
the MAP estimator in $X$ is defined as the minimizer of the Onsager-Machlup functional (OMF)~\cite{dashti_map_2013,Li20151}:
\begin{equation}
 I(u) := \Phi(u)+\frac12\|u\|^2_E,\label{e:OM}
\end{equation}
over the Cameron-Martin space $E$. In Section~\ref{s:examples}, we shall use OMF as an indicating quantity to compare the performance of various MCMC algorithms.
Finally we quote the following lemma~(\cite{da2006introduction}, Chapter~1), which will be useful in next section:
\begin{lemma}
There exists a complete orthonormal \label{lm:eigens}
basis $\{e_k\}_{k\in\N}$ on $X$ and a sequence of non-negative numbers $\{\alpha_k\}_{k\in\N}$
such that ${C_0} e_k = \alpha_k e_k$ and $\sum_{k=1}^\infty \alpha_k <\infty$, i.e., 
 $\{e_k\}_{k\in\N}$ and $\{\alpha_k\}_{k\in\N}$ being the eigenfunctions and eigenvalues of $C_0$ respectively.
\end{lemma}

Without loss of generality, we assume that the eigenvalues $\{\alpha_k\}_{k=1}^\infty$ are in a descending order.

\section{Gaussian mixture based independence sampler} \label{s:mixture}
In this section, we present our Gaussian mixture based independence sampler and show that it is well-defined in the function space.

\subsection{Independence sampler MCMC}
We start by briefly reviewing the independence sampler MCMC algorithm. 
Given a proposal distribution $\mu$, 
we define measures 
\begin{gather*}
\nu(du,du') = \mu(du') \mu^y(du),\\
\nu^\dagger(du,du') = \mu(du) \mu^y(du'),
\end{gather*}
  on the product space $X\times X$.
When $\nu^\dagger$ is absolute continuous with respect to $\nu$, we can define the acceptance probability~\cite{tierney1998note}
\begin{equation} 
A(u,u') = \min
\left \{1, 
\frac{d\nu^\dagger}{d\nu}(u,u')
 \right \},
\label{acceptance} 
\end{equation}
where 
\begin{equation}
\frac{d\nu^\dagger}{d\nu}(u,u') = \frac{d\mu^y}{d\mu}(u') \frac{d\mu}{d\mu^y}(u).\label{e:accept}
\end{equation}
The IS MCMC in a function space proceeds as follows in each iteration:
\begin{enumerate}
\item Draw a sample $u_\mathrm{proposed}$ from the proposal $\mu$.
\item Let $u_\mathrm{next} = u_\mathrm{proposed}$ with probability $A(u_\mathrm{current},u_\mathrm{proposed})$ and\\ $u_\mathrm{next} = u_\mathrm{current}$ with probability $1-A(u_\mathrm{current},u_\mathrm{proposed})$.
\end{enumerate}

We reinstate that the function space IS algorithm  is well-defined if and only if  $\nu^\dagger$ is absolutely continuous with respect to $\nu$,
which requires that $\mu$ and $\mu^y$ are equivalent to each other. 
Since $\mu^y$ and $\mu_0$ are equivalent, it suffices to require $\mu$ and $\mu_0$ to be equivalent. 
Interestingly, the pCN scheme with a specific choice of parameter values yields a dimension-independent IS whose proposal distribution is simply the prior.   
Despite its dimension-independence property, simply proposing according to the prior is inefficient when the data is highly informative, i.e., the posterior being far apart from the prior. Next we shall introduce a more efficient proposal measure than the prior that is to be used in IS MCMC algorithms.

\subsection{Gaussian mixture proposals}
In finite dimensional Bayesian inference problems, Gaussian mixture (GM) distributions~\cite{mclachlan2004finite} are often used as the IS proposal distributions for their flexibility and convenience to draw samples from. 
We now extend the use of GM to the infinite dimensional setting.
Let $\{\mu_j\}_{j=1}^J$ be a set of Gaussian measures on $X$ with $\mu_j = N(m_j,C_j)$ for $j=1...J$,
and we define the Gaussian mixture proposal as
\begin{equation}
\mu(dx) = \sum_{j=1}^J w_j\mu_j (dx) \label{e:mu}
\end{equation}
where $\{w_j\}_{j=1}^J$ are the mixing weights with $\sum_{j=1}^J w_j =1$.
It should be clear that $\mu$ is equivalent to $\mu_0$ as long as each $\mu_j$ is equivalent to $\mu_0$,  
and moreover the Radon-Nikodym derivative of $\mu$ to $\mu_0$ is 
\begin{equation}
\frac{d\mu}{d\mu_{0}}(u)=\sum_{j=1}^{J}w_{j}\frac{d\mu_{j}}{d\mu_{0}}(u).\label{e:dmudmu0}
\end{equation}
Next we discuss our parametrization of each $\mu_i$.  
First recall that, according to Lemma~\ref{lm:eigens},  $\{e_k\}_{k\in\N}$ form a complete basis set of $X$.
Our parametrization of $\mu_i$ is in the form of:
\begin{subequations}\label{e:prop}
\begin{gather}
m_j = \sum_{k=1}^\infty x_{j,k}{\alpha_k}e_k,\\
C_j^{-1} = C_0^{-1} + H_j
\end{gather}
where each $H_j$ is defined as 
\begin{equation}
 H_j \,\cdot = \sum_{k=1}^\infty h_{j,k}\<e_k,\cdot\>e_k
\end{equation}
\end{subequations}
and $x_{j,k}$ and $h_{j,k}$ are coefficients.  The following Theorem provides a sufficient condition for 
$\mu_j = \@N(m_j,C_j)$ to be a well defined Gaussian measure on $X$  
 and equivalent to $\mu_0$.
\begin{theorem} \label{th:muj}
If $x_j, h_j\in {l}_2$, and $h_{j,k} > -\frac{1}{\alpha_k}$ for all $k\in\N$, $\mu_j =\@N(m_j,C_j)$ is a Gaussian measure on $X$ that is equivalent to $\mu_0$. 
\end{theorem}
\begin{proof}
We let $\{\beta_{j,k}\}_{k\in\N}$ be the eigenvalues of $C_j$, i.e, $C_j e_k = \beta_{j,k}e_k$ for all $k\in\N$. And it is easy to see that,
\begin{equation}
\beta_{j,k} = (\alpha_k^{-1}+h_{j,k})^{-1}=\frac{\alpha_k}{1+\alpha_k h_{j,k}}. \label{e:beta}
\end{equation}
As $x_j, h_j\in l_2$, $\frac{1}{1+\alpha_k h_{j,k}}$ is bounded and thus $\sum_{k=1}^{\infty}\beta_{j,k}<\infty$.
It follows that $C_j \in L_1^{+}(X)$ and $\mu_j =\@N(m_j,C_j)$ defines a Gaussian measure on $X$.

We now show that $\mu_j$ is equivalent to $\mu_0$. First we introduce $\mu'_j = \@N(0, C_j)$.
Using Eq.~\eqref{e:beta} and $h_{j,k} > -\frac{1}{\alpha_k}$ for all $k\in\N$, we can get 
$$\sum_{k=1}^{\infty}\frac{(\beta_{j,k}-\alpha_k)^2}{(\beta_{j,k}+\alpha_k)^2}=\sum_{k=1}^{\infty}\frac{\alpha_k^2h_{j,k}^2}{(2+\alpha_k h_{j,k})^2}
\leq \sum_{k=1}^{\infty} \alpha^2_kh_{j,k}^2<\infty,$$
as $\lim_{k\rightarrow\infty} \alpha_k = 0$ and $h_j\in l_2$.
 By the Feldman-Hajek theorem~\cite{da2006introduction}, we have $\mu'_j$ is equivalent to $\mu_0$.
Now recall that $m_j \in E = C_0^{\frac{1}{2}}(X)=C_j^{\frac{1}{2}}(X)$, and so we have $\mu'_j=\@N(0, C_j)$ and $\mu_j=\@N(m_j, C_j)$ are equivalent, 
which completes the proof.
\end{proof}

Let us assume for now that the conditions in Theorem \ref{th:muj} is satisfied and we shall verify this assumption later. 
It is easy to show that
\begin{multline}
\frac{d\mu_j}{d\mu_{0}}(u)= 
\frac{|C_{0}|^{1/2}}{|C_{j}|^{1/2}}\exp(-\frac{1}{2}\|C_{j}^{-1/2}m_{j}\|_X^{2}+\<u, C_j^{-1}m_j\>_X-\frac{1}{2}\<u, H_{j}u\>_X)\\
=\prod_{k=1}^{\infty}\sqrt{\frac{\alpha_k}{\beta_{j,k}}}\,\exp\bigg[-\frac{1}{2}\sum_{k=1}^{\infty}
\left(\frac{\alpha_k^2}{\beta_{j,k}}x_{j,k}^{2}+h_{j,k}u_{k}^{2} -\frac{2\alpha_k}{\beta_{j,k}}x_{j,k}u_{k}\right)\bigg], \label{e:mujmu0}
\end{multline}
where  $u_k = \<u, e_k\>$ is the projection of $u$ onto $e_k$.
Note that the density ${d\mu_j}/{d\mu_{0}}$ actually depends on $m_j$ and $h_j$, and thus for convenience's sake, 
we define a function $f(\cdot,\cdot,\cdot)$ such that
\[ f(u,x_j,h_j) = \frac{d\mu_j}{d\mu_0}(u),\]
and we then can derive from Eq~\eqref{e:dmudmu0} that  
\[
\frac{d\mu^{y}}{d\mu}(u)=\frac1Z\exp(-\Phi(u))/(\sum_{j=1}^{J}w_{j}f(u,x_{j},h_{j})),\]
and the density $d\mu/d\mu_y$ can be computed accordingly.

\subsection{Minimizing the Kullback-Leibler divergence}
Now recall that for the algorithm to be efficient we need the proposal $\mu$ to be close to $\mu^y$ and a natural choice is to determine $\mu$ by minimizing the 
Kullback-Leibler divergence (KLD) between $\mu^y$ and $\mu$:
\begin{equation}
D_{KL}(\mu^y||\mu)=\int \log\frac{d\mu^y}{d\mu}(u)\mu^y(du),\label{e:kld}
\end{equation}
where $\mu$ is parametrized with Eq.~\eqref{e:prop}.
Note that $x_j$ and $h_j$ are set to be of infinite dimensions in the formulation above.
In numerical simulations, however,  $x_j$ and $h_j$ must be truncated at some finite number $K$.
Such a truncation is also reasonable from a practical point of view. 
In fact, one often can realistically assume that the data is only informative on a finite number of directions~\cite{cui2014dimension,cotter2013mcmc} in $X$, 
and under this assumption, we only need to keep a finite number of components of each $x_j$ and $h_j$.
We emphasize that $K$ which represents the number of dimensions that are informed by the data (i.e., the so-called intrinsic dimensionality), 
should not be confused with the discretization dimensionality of the problem, i.e., the number of mesh points used to represent the unknown.  
Determining the value of $K$ is an important task for our algorithm and here we choose $K$ with a heuristic approach: 
\[K = \min\{k\in\N~|~ \frac{ \alpha_k}{\alpha_1} < \epsilon\},\] where $\epsilon$ is a prescribed threshold.  
 In what follows, we shall adopt this finite, $K$-dimensional formulation, 
and thus we have the following optimization problem:
\begin{equation}
\min_{\{x_j,\,h_j\in \R^K, \,w_j\in [0,1]\}_{j=1}^J} D_{KL}(\mu^y||\mu), \label{e:min}
\end{equation}
subject to $\sum_{j=1}^J w_j = 1$. 
By some elementary calculations, we can show that Eq.~\eqref{e:min} is equivalent to  
\begin{equation}
\min_{\{x_j,\,h_j\in \R^K, \,w_j\in [0,1]\}_{j=1}^J}-\int \log[\sum_{j=1}^{J}w_{j}f(u,x_{j},h_{j})]\mu^y(du), \label{e:max}
\end{equation}
subject to $\sum_{j=1}^J w_j = 1$.
We now show that the proposal $\mu$ constructed this way is well-defined in function space, and to this end we have the following corollary. 
\begin{corollary}
If $\{x_j,h_j,w_j\}_{j=1}^J$ is a solution of Eq.~\eqref{e:max} , the resulting $\mu$ is equivalent to $\mu_0$.
\end{corollary}
\begin{proof}
It is obvious that if $\{x_j,h_j,w_j\}_{j=1}^J$ is a solution of Eq.~\eqref{e:max}, $x_j,h_j\in l_2$. 
Taking partial derivative of the objective function in Eq.~\eqref{e:max} with respect to $h_{j,k}$ and setting it to be zero yields the following equation:
\[
\int\frac{w_{j}f(u, x_{j}, h_{j})}{\sum_{l=1}^{J}w_{l}f(u, x_{l}, h_{l})}d\mu^{y}\frac{\alpha_k}{1+\alpha_k h_{j,k}}
=\int\frac{w_{j}f(u, x_{j}, h_{j})(\alpha_k x_{j,k} - u_k)^2}{\sum_{l=1}^{J}w_{l}f(u, x_{l}, h_{l})})d\mu^{y} .
\]
As the following two integrals are obviously positive:
\[
\int\frac{w_{j}f(u, x_{j}, h_{j})}{\sum_{l=1}^{J}w_{l}f(u, x_{l}, h_{l})}d\mu^{y}>0, \quad \mathrm{and}\quad
\int\frac{w_{j}f(u, x_{j}, h_{j})(\alpha_k x_{j,k} - u_k)^2}{\sum_{l=1}^{J}w_{l}f(u, x_{l}, h_{l})})d\mu^{y} >0,
\]
we have $1+\alpha_k h_{j,k}>0$.
Thus all the conditions of Theorem~\ref{th:muj} are satisfied and the corollary follows immediately. 
\end{proof}

Finally we note that, 
in the special case where J=1, namely the proposal being simply a Gaussian distribution,
our parametrization is similar to the finite rank representation used in \cite{pinski2013kullback,pinski2014algorithms}. 
In fact, the aforementioned works also proposed to approximate the posterior with a Gaussian distribution by minimizing the KLD between the two distributions. 
The major difference is the KLD (recall that it is asymmetric) formulation: the authors of \cite{pinski2013kullback,pinski2014algorithms} compute
 the divergence from the Gaussian approximation to the true posterior, while here we compute the divergence the other way around.
An advantage of the present formulation is that 
the solution to Eq.~\eqref{e:max} can be explicitly obtained:
\begin{subequations}
\label{e:one}
\begin{gather}
x_{k} = \frac{1}{\alpha_k}\int u_kd\mu^{y},\\
h_k =\frac1{\int(\alpha_k x_{k}-u_k)^2d\mu^y} - \frac1{\alpha_k},
\end{gather}
\end{subequations}
for $k = 1...K$, while in their formulation the resulting optimization problem has to be solved with a stochastic optimization algorithm.
The explicit solutions~\eqref{e:one} are of essential importance in our adaptive algorithm.

\section{The adaptive algorithm}\label{s:algorithm}
In  this section we discuss the algorithm to implement the IS method proposed in Section~\ref{s:mixture},
starting with an introduction to the adaptive MCMC. 
\subsection{Adaptive MCMC}
The basic idea of the adaptive MCMC is to repeatedly adjust the proposal parameters using the information in the previous samples.  
Here we are focused on the adaptive algorithms with IS~\cite{holden2009adaptive,giordani2010adaptive,keith2008adaptive,gaasemyr2003adaptive}, while noting that other types 
of adaptive algorithms include the adaptive MH~\cite{haario2001adaptive}, the adaptive MALA~\cite{atchade2006adaptive,marshall2012adaptive}, 
and the adaptive Metropolis-within-Gibbs~(MwG)~\cite{roberts2009examples}. Specifically our adaptive algorithm has the following three key ingredients. 
First, to enforce the asymptotic ergodicity, we terminate the adaptation in a finite number of steps. 
Secondly we use a tempered pre-run to obtain the initial parameter values for the iteration.
Simply speaking the technique of tempering is to construct a sequence of intermediate distributions that converge to the true posterior $\mu^y$ 
and use these intermediate distributions to guide the MCMC samples to the true posterior. 
This strategy is particularly useful for multimodal posterior distributions.  
Without loss of generality, we assume that the tempering distributions are augmented by a tempering parameter $\lambda$:
\[\frac{d\mu^{y,\lambda}}{d\mu_0} \propto \exp(-\lambda\Phi(u)),\]
and clearly $\mu^{y,\lambda} = \mu^y$ when $\lambda =1$ and the tempering distribution is ``wider'' than the true posterior for $0\leq\lambda <1$.
In practice we can choose a finite number of tempering parameters $\{\lambda_i\}_{i=1}^{I_\mathrm{temp}}$ where 
$0\leq\lambda_1<\lambda_2<...< \lambda_{I_\mathrm{temp}}=1$.
We also note that for problems where the posterior is not too far apart from the prior, tempering may not be necessary. 
Finally we estimate and update the proposal parameters after every fixed number of iterations. 
The adaptive scheme is summarized as the following:
\begin{itemize}
\item Initialization: the total number of iterations $I_\mathrm{tol}$, the number of adapted iterations $I_\mathrm{adp}$, the number of pre-run (tempering) iterations $I_\mathrm{temp}$,
a set of  tempering parameters $\{\lambda_i\}_{i=1}^{I_\mathrm{temp}}$, 
the number of samples used in each tempered iteration $N_\mathrm{temp}$, the number of samples in each iteration $N_S$. 
\item Pre-run (optional): let $\mu'_{(0)} =\mu_0$; for $i =1:I_\mathrm{temp}$ perform:
\begin{enumerate}
\item Run MCMC with proposal $\mu'_{(i-1)}$ to draw a set of $N_\mathrm{temp}$ samples from $\mu^{y,\lambda_i}$, denoted by $S'_i$.  
 

\item Update the parameter values with samples $S'_i$ obtaining proposal $\mu'_{(i)}$;

\end{enumerate}

\item Iteration: let  $S = \emptyset$ and  $\mu_{(0)}= \mu'_{(I_\mathrm{temp})}$; for i= $1$ to $I_\mathrm{tol}$ perform:
\begin{enumerate}
\item Run MCMC with proposal $\mu_{(i-1)}$ to draw a set of $N_S$ samples from $\mu^y$, denoted by $S_i$.  
Let $S = S \cup S_i$.  

\item If $i<I_\mathrm{adp}$, update the parameter values with samples $S$ obtaining proposal $\mu_{(i)}$;
otherwise, let $\mu_{(i)} = \mu_{(i-1)}$. 
\end{enumerate}
\end{itemize}

The adaptive algorithm presented above is rather simple; we note, however, that our method is rather flexible and one can pair it with any desired adaptive IS algorithm. 
A key step in the adaptive algorithm is to estimate the parameters from the samples, 
which is done by solving the sample average estimator of the optimization problem~\eqref{e:max}: 
\begin{equation}
\max_{\{x_j,h_j,w_j\}_{j=1}^J}\sum_{n=1}^N \log[\sum_{j=1}^{J}w_{j}f(u^{n},x_{j},h_{j})], \label{e:max2}
\end{equation}
subject to $\sum_{j=1}^J w_j = 1$.
Next we discuss two methods to solve Eq.~\eqref{e:max2}.

\subsection{Expectation Maximization algorithm}
The expectation maximization (EM) is one of the most popular methods to determine the parameters in mixture models~\cite{mclachlan2004finite}. 
Simply put, the EM algorithm iteratively updates the parameter values in a way that the function value is always increased until convergence is achieved. Each iteration consists of an Expectation-step and a Maximization-step.
It should be noted that, the EM algorithm, is not guaranteed to converge to the optimal solutions in general~\cite{wu1983convergence}. 
The theory and implementation details of the EM algorithm and its application to mixture models can be found in the aforementioned references , and we shall not repeat them here. 
When applied to our problem,  the update formula in each iteration can be explicitly obtained. 
In the Expectation-Step, the membership probability $q^n_{j}$, namely the probability that a sample $u^{n}$ is in the mixture $j$, is computed, 
\begin{equation}
q_{j}^n = \frac{w_{j}f(u^{n}, h_{j}, m_{j})}{\sum_{j=1}^{J}w_{j}f(u^{n}, h_{j}, m_{j})},
\end{equation}
for each $j=1...J$ and $n=1...N$; in the 
Maximization-Step, the parameter values are updated using the following equations:
\begin{subequations}
\begin{align}
&w_{j}=\frac{1}{N}\sum_{i=1}^{N} q^{i}_j,\\
&x_{j,k}=\frac1{N\alpha_kw_j}\sum_{n=1}^{N}
q^n_j u_{k}^{n}\\
&h_{j,k}=\sum_{n=1}^{N}q^n_j(\sum_{n=1}^{N}q^n_j(\alpha_k x_{j,k} - u_k^{n})^2)^{-1}-\frac{1}{\alpha_{k}},
\end{align}
\end{subequations}
where $u^n_k = \<u^n,e_k\>$.
  The EM algorithm is arguably the most common method to estimate the parameters of mixtures. 
	However, our numerical tests indicate that in some practical problems the EM algorithm is not sufficiently reliable especially when the sample set only contains a small number of accepted draws.  
	Moreover, our algorithm frequently updates the proposal parameters, which makes the computationally intensive EM algorithms less
	attractive from an efficiency perspective. For these reasons, we propose an alternative method to EM, which estimates the mixture parameters using clustering. 

\subsection{Estimating parameters with clustering}

Our estimation method with clustering is largely based on the finite dimensional method developed in \cite{giordani2010adaptive}.
The idea is rather simple: one first partitions the samples into several clusters and then fit each cluster with a Gaussian distribution. 
A difficulty here is that our MCMC samples are of infinite dimension, which makes clustering challenging. 
To solve the problem, we first project the samples onto the $K$ eigenfunctions of the covariance operator and then cluster the resulting $K$ dimensional data $\{ (u_1^n,...,u_K^n)\}_{n=1}^N$ and $u_k^n =\<u^n,e_k\>$. 
Specifically we use the k-means algorithm to cluster the data,  
and the number of clusters $J$ is determined with the Bayesian information criteria~(BIC) method~\cite{mclachlan2004finite}. In fact we have found in our numerical tests that the algorithm is rather robust against the number of clusters. 
We then use the Gaussian distribution parametrized in the form of Eq~\eqref{e:prop} to fit each cluster, and thanks to Eq.~\eqref{e:one},
 the parameters values can be estimated explicitly as,
 \begin{subequations}
\label{e:x-h}
\begin{gather}
x_{j,k} = \frac{1}{N_j\alpha_k}\sum_{u^n\in \Theta_j}u_{k}^{n},\\
h_{j,k} = \frac{1}{\frac{1}{N_j}\sum_{u^n\in\Theta_j}(u_{k}^{n})^2-m_{j,k}^{2}}-\frac{1}{\alpha_k},
\end{gather}
\end{subequations}
where $\Theta_j$ is the $j$-th cluster of samples, $N_j$ is the sample size of $\Theta_j$, 
for $j=1...J$ and $k=1...K$. 
The mixture weights are simply determined by the fraction of samples in each cluster. 
We note that the clustering based method does not generally yield a solution to Eq.~\eqref{e:max2} and thus we regard it as an approximate method to estimate the parameters.  
We conclude the section with a pseudo code (Algorithm \ref{a:code}) of our algorithm, and interested readers can use it as a basis for their own implementation.

\begin{algorithm}
\SetAlgoLined
\SetKwInOut{Input}{input} \SetKwInOut{Output}{output} 
\Input{$I_\mathrm{temp}$,  $\{\lambda_i\}_{i=1}^{I_\mathrm{temp}}$, $N_\mathrm{temp}$, $N_\mathrm{tol}$, $N_\mathrm{\max}$,  $N_\mathrm{adp}.$}
\Output{$N_\mathrm{tol}$ samples drawn from $\mu^y$: $\{u^n\}_{n=1}^{N_\mathrm{tol}}$.}
  $\mu \la \mu_0$;\\
 \For {$i\la 1$ \KwTo $I_\mathrm{temp}$}{
      draw $u^0 \sim \mu$; \\
\For {$n\la 1$ \KwTo $N_\mathrm{temp}$} {
 draw $u'\sim\mu$;\\
 draw $a \sim U[0,1]$ and compute
\[A \la \min\{1,\frac{d\mu^{y,\lambda_i}}{d\mu}(u') \frac{d\mu}{d\mu^{y,\lambda_i}}(u^{n-1})\};\]
 \\
 \leIf{$A>a$} {$u^{n} \la u'$;}{$u^{n}\la u^{n-1}$}
 }
cluster $\{u^0,...,u^{N_\mathrm{temp}}\}$ into $J$ subsets:  $\Theta_1,...,\Theta_J$ ;  \\
\For{$j\la 1$ \KwTo $J$}{
$N_j\la$ sample size of $\Theta_j$, $w_j \la N_j/N_\mathrm{temp}$;\\
compute parameters $x_j$ and $h_j$ using Eqs.~\eqref{e:x-h};\\
compute $\mu_j$ using Eq.~\eqref{e:mujmu0};}
 $\mu \la\sum_{j=1}^Jw_j \mu_j$;
}
draw $u^0\sim \mu$; \\
\For {$n\la 1$ \KwTo $N$} {
 draw $u' \sim\mu$;\\
 draw $a \sim U[0,1]$ and compute
\[A \la \min\{1,\frac{d\mu^{y}}{d\mu}(u') \frac{d\mu}{d\mu^{y}}(u^{n-1})\};\]
 \leIf{$A>a$} {$u^{n} \la u'$;}{$u^{n}\la u^{n-1}$}
\If{$(n<N_{\max})\&(n \mod N_\mathrm{adp} = 0)$}{
cluster $\{u^0,...,u^{n}\}$ into $J$ subsets:  $\Theta_1,...,\Theta_J$ ;  \\
\For{$j\la 1$ \KwTo $J$}{
$N_j\la$ sample size of $\Theta_j$, $w_j \la N_j/n$;\\
compute parameters $x_j$ and $h_j$ using Eqs.~\eqref{e:x-h};\\
compute $\mu_j$ using Eq.~\eqref{e:mujmu0};}
 $\mu \la\sum_{j=1}^Jw_j \mu_j$;}}
\caption{The complete algorithm for the adaptive IS with GM. $I_\mathrm{temp}$ 
is the number of tempered iterations. 
$\{\lambda_i\}_{i=1}^{I_\mathrm{temp}}$ are the tempering parameters.
$N_\mathrm{temp}$ is the number of samples used in each tempered iteration.
$N_\mathrm{tol}$ is the total number of samples drawn by the algorithm. 
$N_\mathrm{adp}$ is the number of samples drawn between two consecutive  parameter updates. $N_\mathrm{\max}$ is the maximum length of chain before the adaptation is terminated. }
\label{a:code}
\end{algorithm}

\section{Numerical examples} \label{s:examples}
\subsection{An ordinary differential equation example}
Our first example is a simple inverse problem where the forward model is governed by an ordinary differential equation (ODE):
\begin{equation}
\frac{d x(t)}{dt} = -u(t) x(t)
\end{equation}
with a prescribed initial condition. 
We assume that the solution $x(t)$ is observed  at several times in the interval $[0, T]$ and we want to infer the unknown
coefficient $u(t)$ for $t\in [0, T]$.

\begin{figure}
\centerline{\includegraphics[width=.75\textwidth]{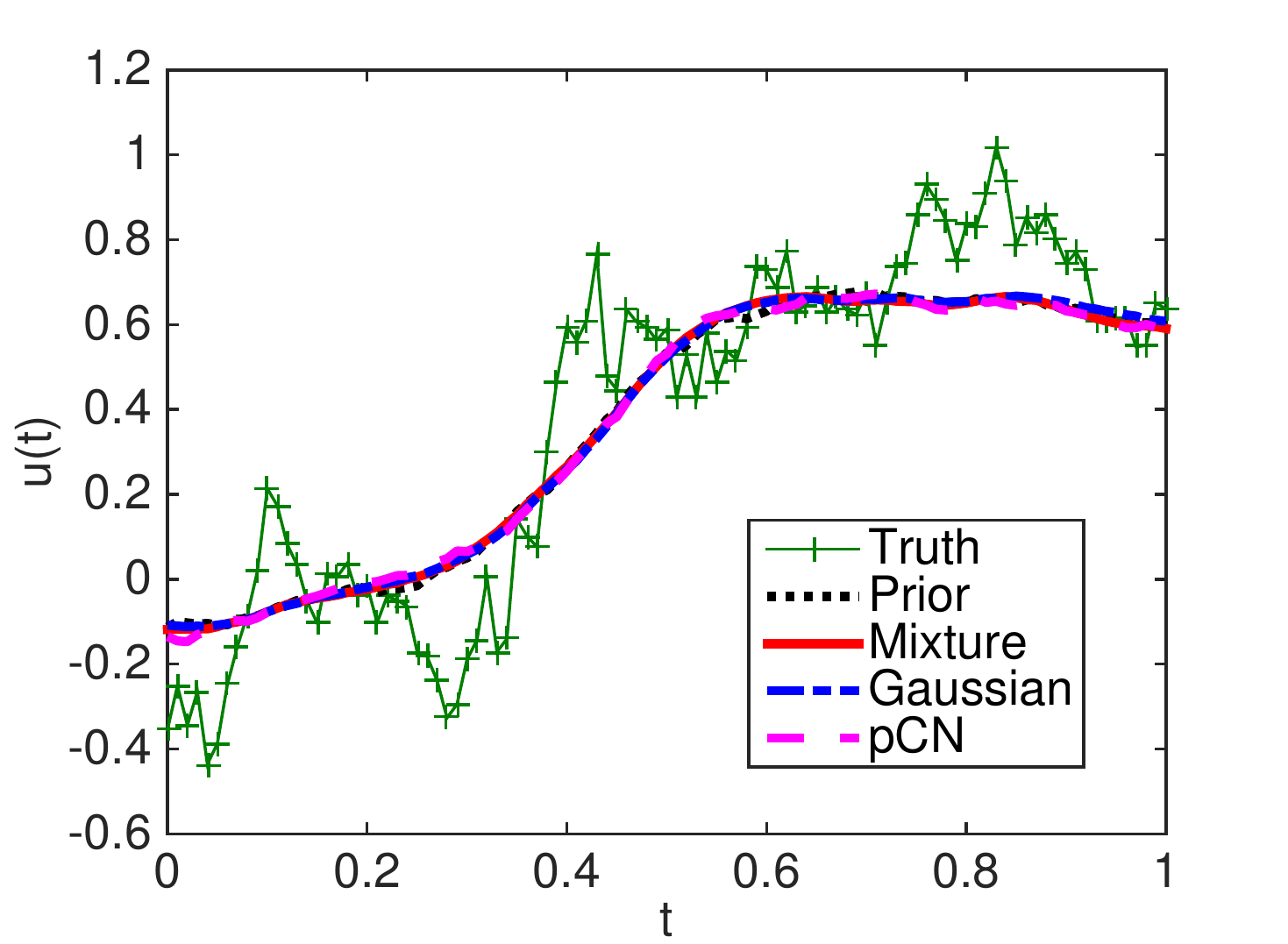}}
\caption{(for example 1) The posterior mean computed with the four different MCMC schemes. The truth is also plotted for comparison. }
\label{f:mean_ode}
\end{figure}

\begin{figure}
\centerline{\includegraphics[width=.9\textwidth]{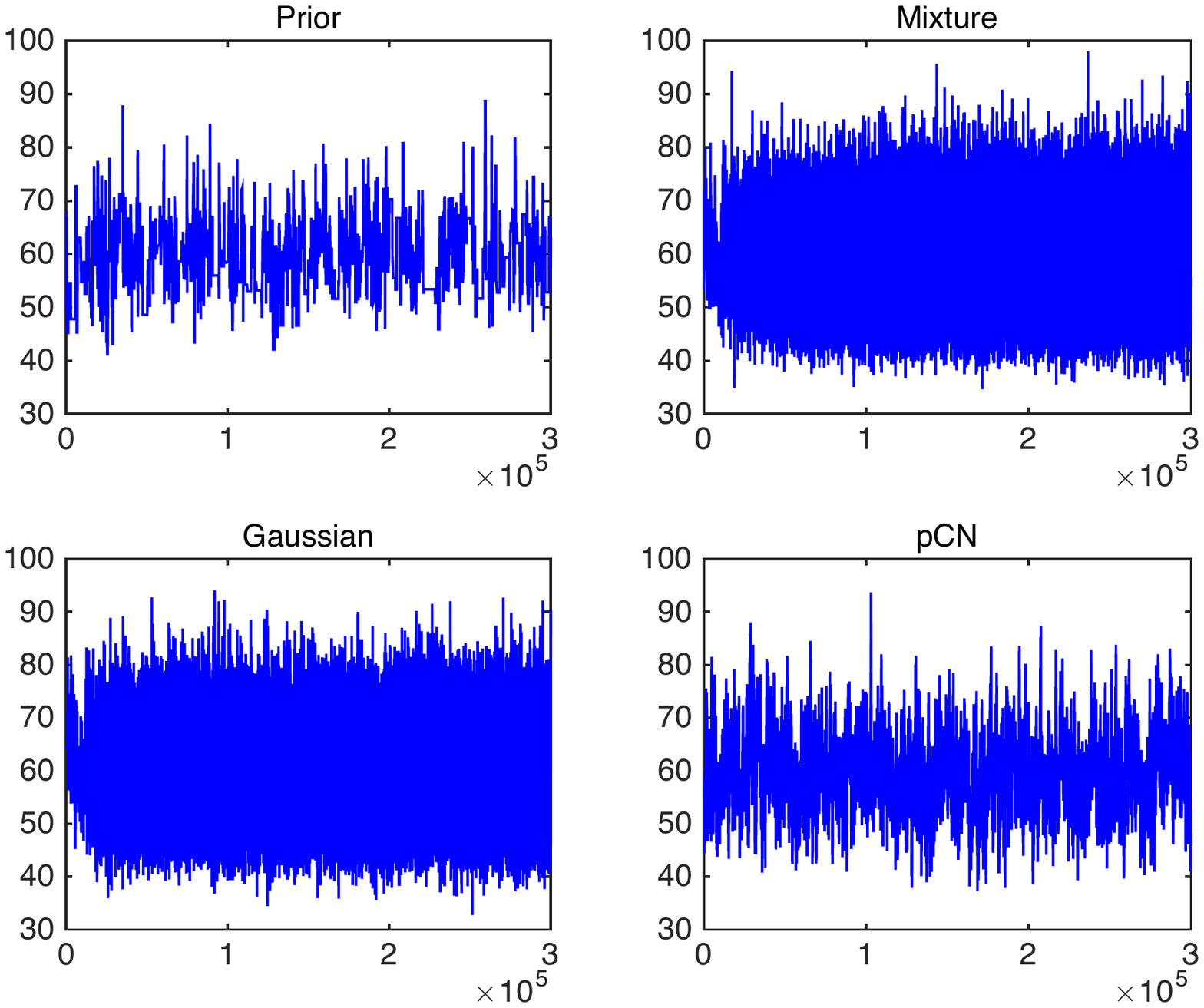}}
\caption{(for example 1) The trace plots of the OMF for the four different MCMC schemes.}
\label{f:trace}
\end{figure}

\begin{figure}
\centerline{\includegraphics[width=.5\textwidth]{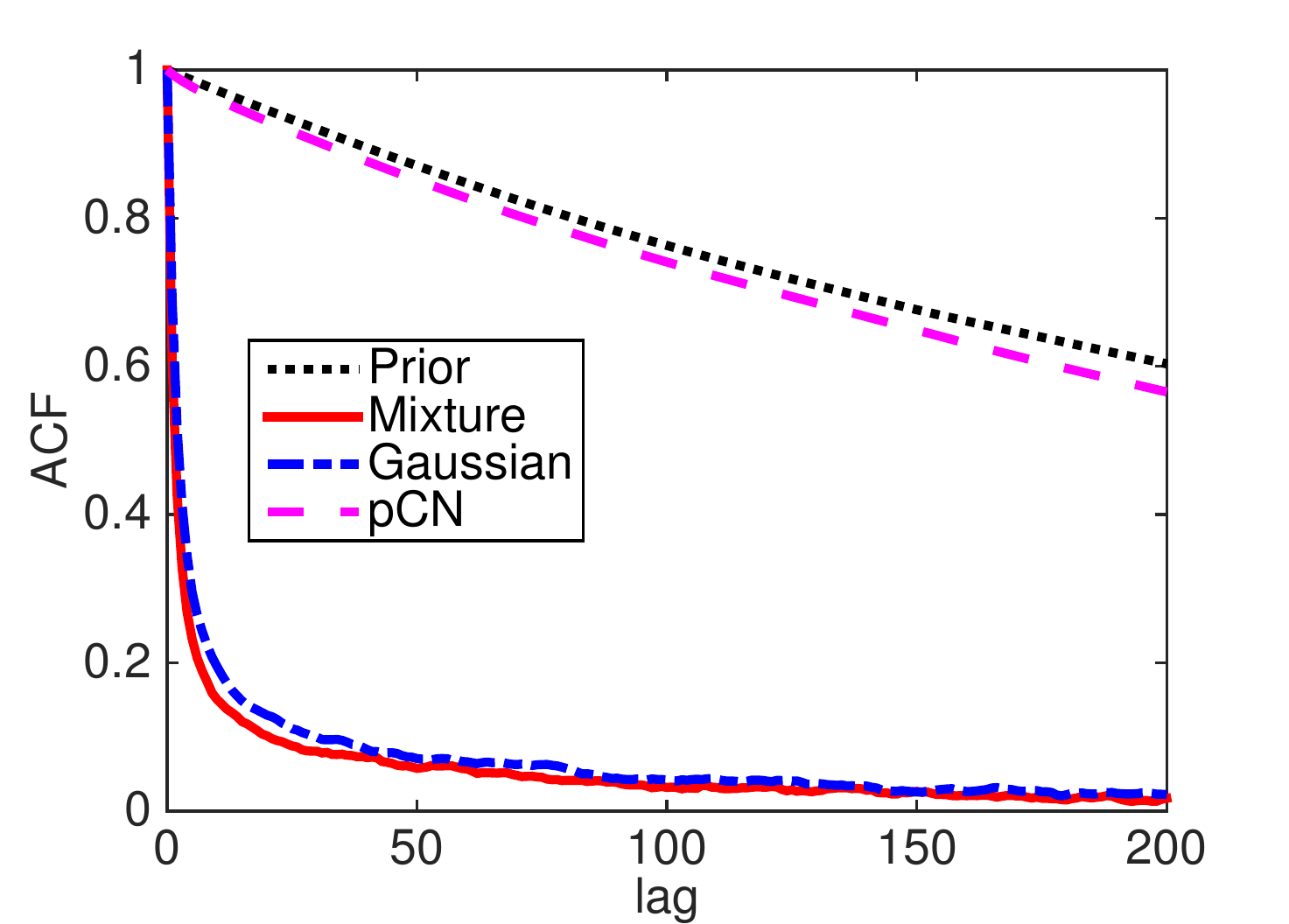}
\includegraphics[width=.5\textwidth]{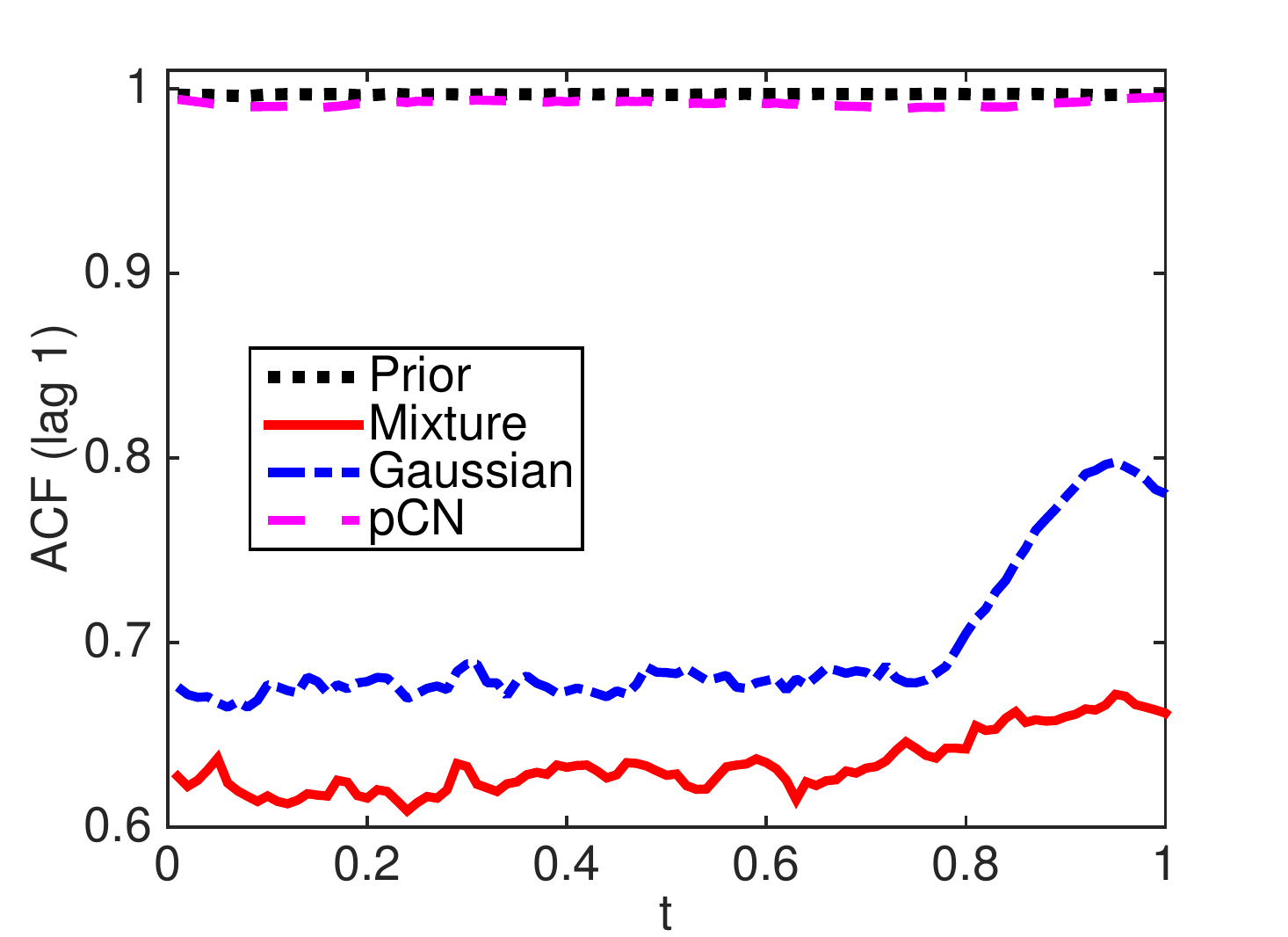}}
\caption{(for example 1) Autocorrelation functions (ACF) for the four different MCMC methods. Left: ACF of the OMF plotted as a function of lags. Right: the lag 1 ACF for $u$ at each grid point.}
\label{f:acf}
\end{figure}

\begin{figure}
\centerline{\includegraphics[width=.5\textwidth]{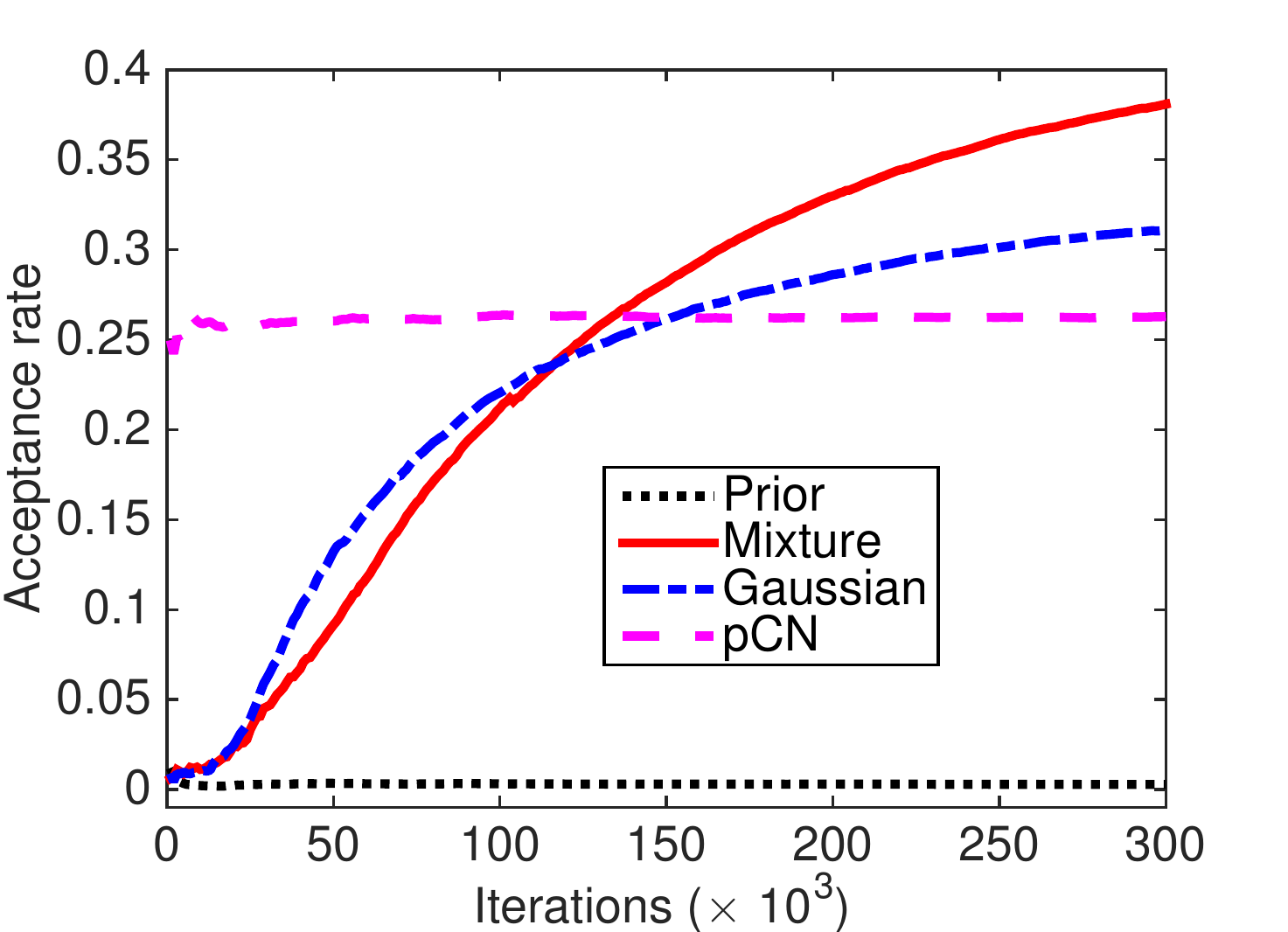}
\includegraphics[width=.5\textwidth]{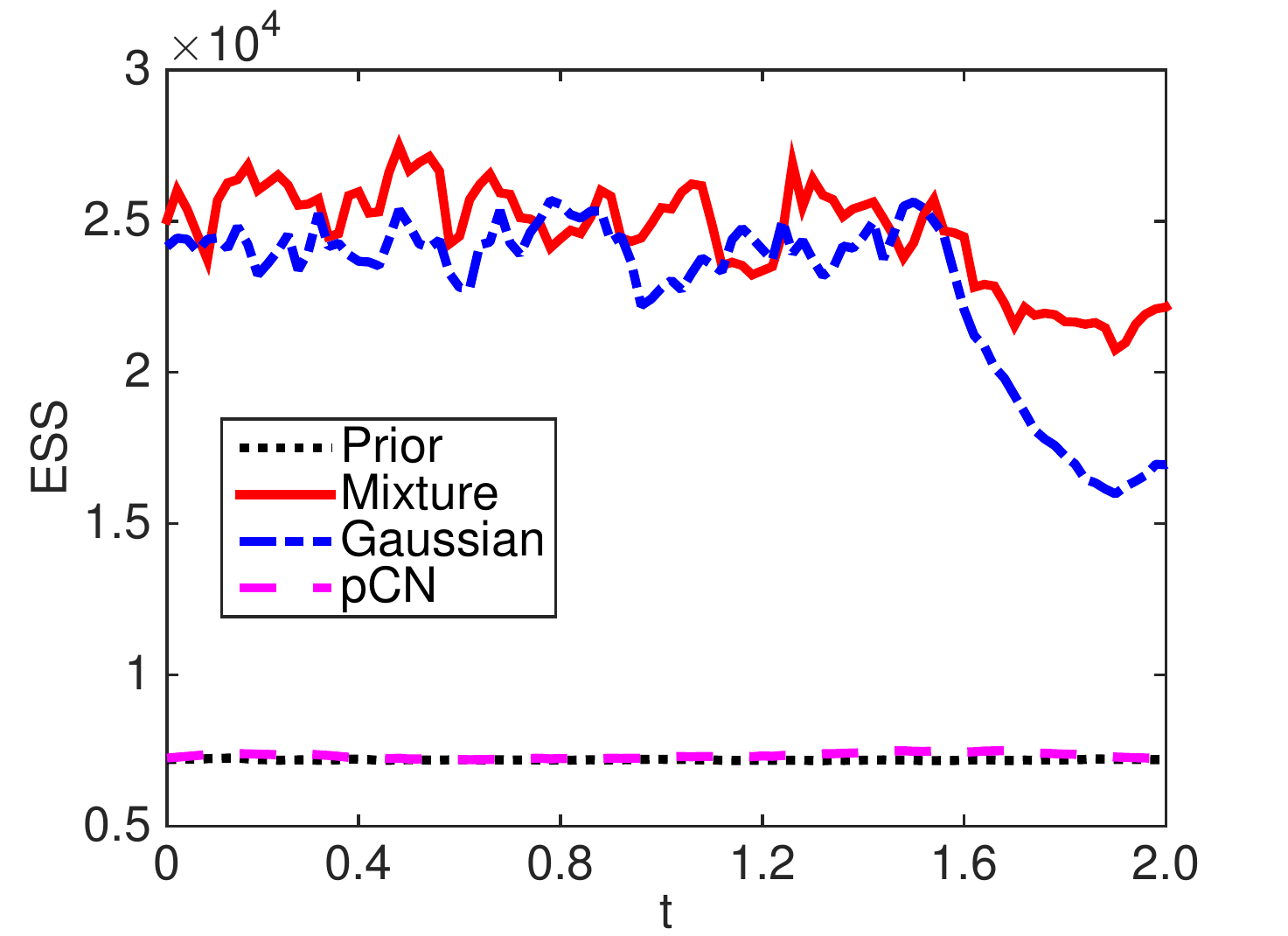}
}
\caption{(for example 1) Left: the acceptance rate of the four MCMC schemes. Right: the ESS at each grid point.}
\label{f:acc-ess}
\end{figure}

In our numerical experiments, we let  the initial condition be $\eta(0) = 1$ and $T=1$. 
Now suppose that the solution is measured every $T/20$ time unit from $0$ to $T$ and the error in each measurement is assumed to be an independent zero-mean Gaussian random variable 
with variance $0.05^2$. In the computation, 100 equally spaced grid points are used to represent the unknown. 
Moreover, we  assume that the state space for $u$ is $X=L_2([0,T])$ and 
 the prior is a zero-mean Gaussian measure in $X$ with an exponential covariance function:
\begin{equation}
C(t,t') = \exp(-\frac{|t-t'|}{2}). \label{e:expcov}
\end{equation} 
The true coefficient $u(t)$ is a realization from the prior (shown in Fig.~\ref{f:mean_ode}) and the data is simulated accordingly. 

We now draw samples from the posterior of $u(t)$ with four different MCMC schemes: prior-based IS, adaptive IS with Gaussian approximation, adaptive IS with Gaussian mixtures, and the random walk pCN (RW-pCN). In each MCMC scheme, $3\times10^{5}$ draws are generated. 
In the prior based IS, one simply proposes according to the prior distribution, and no adaptation is used. 
In the adaptive IS with Gaussian approximation, the proposal is restricted to be a single Gaussian (i.e. $J=1$), and in this case clustering is not needed. 
In both of the adaptive IS methods, the parameters are updated after every 1000 draws, and the parameter adaptation is terminated in the last $10^5$ iterations.
We do not use tempering in this example. The RW-pCN algorithm used in this work iterates as follows
\begin{enumerate}
\item propose   $u_\mathrm{proposed} = \sqrt{1-\beta^2} u_\mathrm{current} + \beta w$, where $w \sim \mu_0$ 
\item 
Let $u_\mathrm{next}=u_\mathrm{proposed}$ with probability 
\[a = \min \{1, \exp( \Phi(u_\mathrm{proposed}) - \Phi(u_\mathrm{current}))\},\] 
and let $u_\mathrm{next} = u_\mathrm{current}$ with probability $1-a$.
\end{enumerate}
In this example we use $\beta =0.1$.
Note that, in all the numerical examples, we choose the stepsize $\beta$ so that the resulting acceptance probability 
is in the range $20\%-30\%$, as is recommended in \cite{roberts2001optimal}.  

In Fig.~\ref{f:mean_ode}, we show the posterior mean computed by the four MCMC schemes, while the truth is also shown for comparison purpose. 
One can see that the results of the four algorithms are nearly identical, suggesting that all the algorithms can estimate the posterior mean to a similar level of accuracy.  We then use the OMF as an indicative parameter and show the trace plots of it in Fig.~\ref{f:trace}.
We see from the plots that the two adaptive IS algorithms achieve much faster mixing rate than the other two methods.   
To further compare the efficiency of the methods, we compute the autocorrelation functions (ACF) of various quantities 
with the samples drawn by the four methods, and plot the ACF results in Fig.~\ref{f:acf}.
In particular, we plot the ACF of the OMF as a function of lag in Fig.~\ref{f:acf} (left) and 
 show the lag 1 ACF for the unknown $u$ at each grid point in  in Fig.~\ref{f:acf} (right).
It can be seen from the figure that, our adaptive algorithms with single Gaussian proposal and with mixtures both result in much lower ACF values than the other two methods.  
When comparing the two adaptive algorithms, the mixture proposal outperforms the single Gaussian.
For the IS algorithms, the acceptance probability is also a useful performance indicator, where higher acceptance rates are usually preferred, while  
it is not the case for random walk algorithms~\cite{roberts2001optimal}. 
In Fig~\ref{f:acc-ess} (left) we plot the acceptance probability as a function  of iterations for all the methods. 
For the three IS algorithms,  one can see that the two adaptive algorithms have significantly higher acceptance probability than the prior based method.
Meanwhile, the acceptance probability of IS with mixtures is higher than that of the one with the single Gaussian. 
The effective sample size (ESS) is another common measure of the sampling efficiency of MCMC~\cite{Kass1998}. 
ESS is computed 
by \[\mathrm{ESS} = \frac{N}{1+2\tau},\]
where $\tau$ is the integrated autocorrelation time and $N$ is the total sample size, and it gives an estimate of the number of effectively independent draws in the chain.
We computed the ESS of the unknown $u$ at each grid point and show the results in Fig.~\ref{f:acc-ess} (right). 
Once again, the plots indicate that the adaptive algorithms produce much more effectively independent samples than the prior based IS
and the RW-pCN, while the mixture proposal outperforms the single Gaussian one in most of the dimensions. 
In summary, in this simple nonlinear inverse problem, we show that our adaptive algorithms are significantly more efficient than
the prior based IS and the RW-pCN. Meanwhile, the mixture proposal outperforms the single Gaussian one, indicating
that the more flexible mixture representation does improve the efficiency.  

\subsection{A bimodal likelihood function example}
Our second example is  an artificially constructed bimodal problem.  Once again we assume the unknown $u\in X = L^2([0,1])$ and the prior is a zero mean Gaussian measure with the same covariance function Eq.~\eqref{e:expcov} as the first example.
We consider a bimodal likelihood function, given by, 
\[ \exp(-\Phi(u)) \propto \exp(-\frac12\|u-\sin(2\pi t)\|_2^2)+\exp(-\frac12\|u+\sin(2\pi t)\|_2^2),\]
and it can be verified that the $\Phi(\cdot)$ chosen this way satisfies the Assumptions (6.1) in ~\cite{cotter2013mcmc}. 
It is easy to see that the posterior distribution should have two modes: one is close to $\sin(2\pi t)$ and the other is close
to $-\sin(2\pi t)$.  

We draw samples from the posterior of $u(t)$ with the same four MCMC schemes used in the first example, and in each MCMC scheme, $5\times10^{5}$ draws are generated. 
In both of the adaptive IS methods, the parameters are updated after every 1000 draws, and the adaptation is terminated in the last $10^5$ iterations, with no tempering used.  In the RW-pCN, we choose $\beta =0.5$. In all the computations, 100 grid points are used to represent the unknown function $u$.

\begin{figure}
\centerline{\includegraphics[width=1\textwidth]{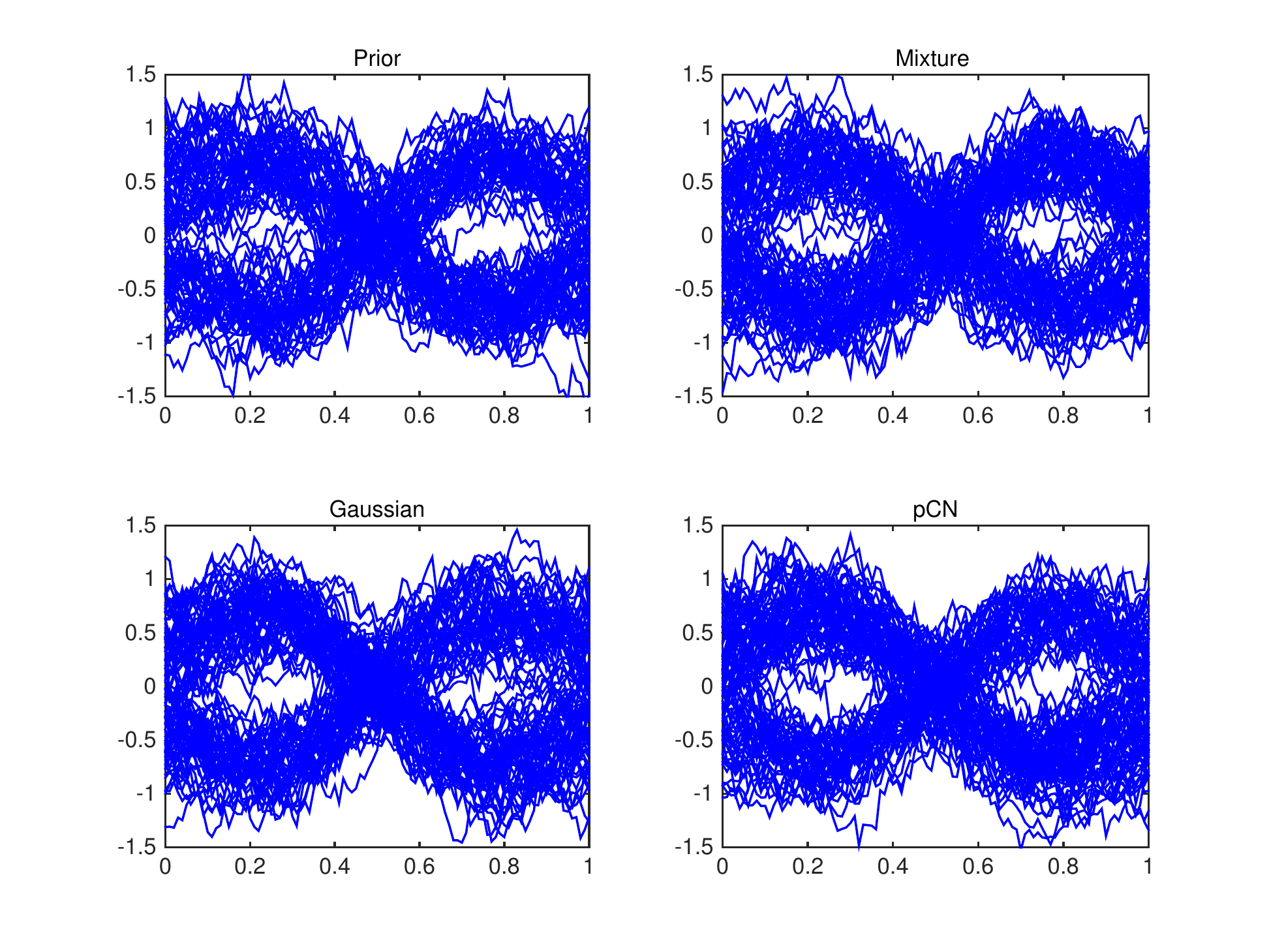}}
\caption{(for example 2) 100 samples randomly selected from the chain drawn by each method.}
\label{f:samples}
\end{figure}

\begin{figure}
\centerline{\includegraphics[width=.9\textwidth]{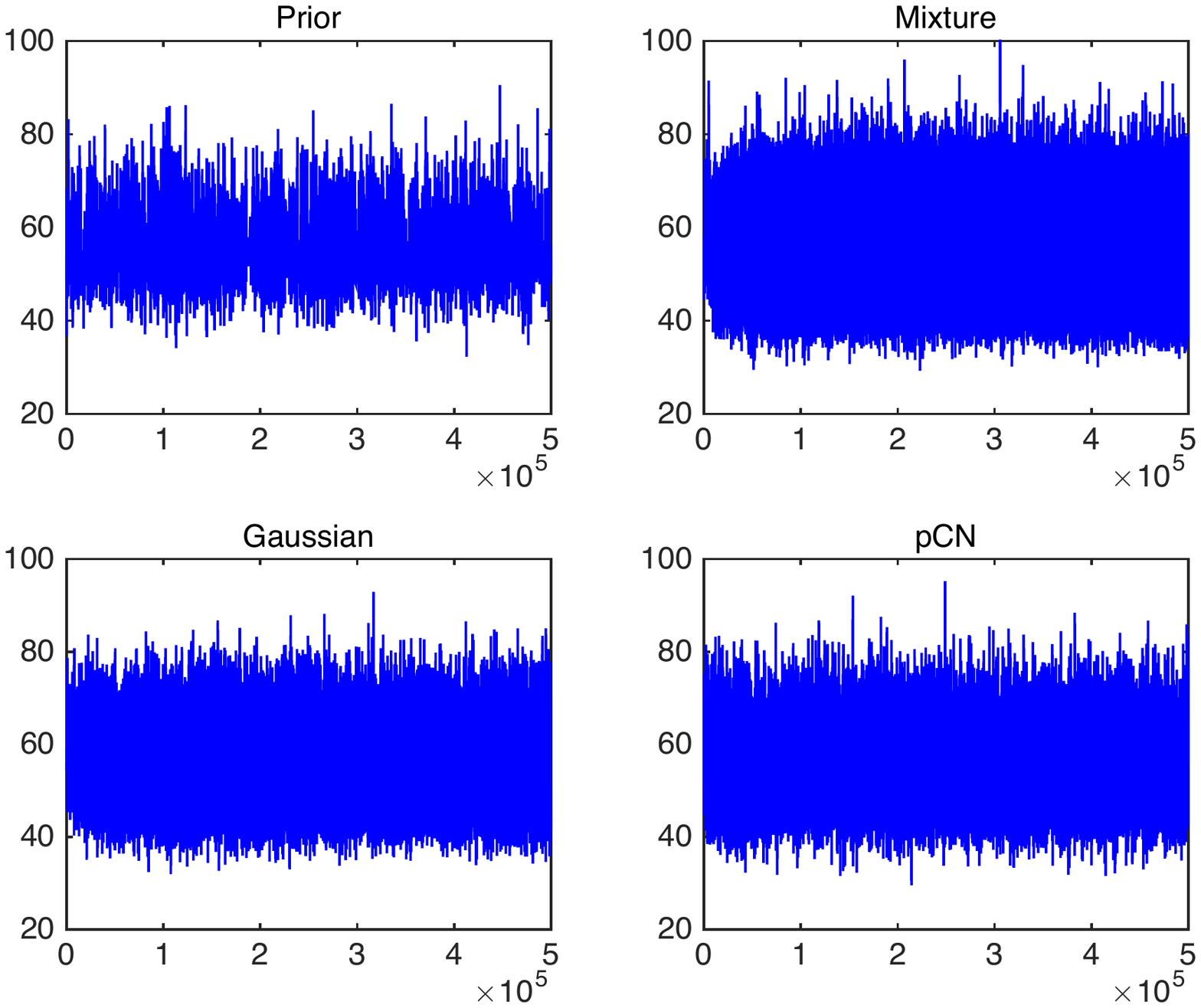}}
\caption{(for example 2) The trace plots of the OMF for the four different MCMC schemes.}
\label{f:trace_sin}
\end{figure}

\begin{figure}
\centerline{\includegraphics[width=.5\textwidth]{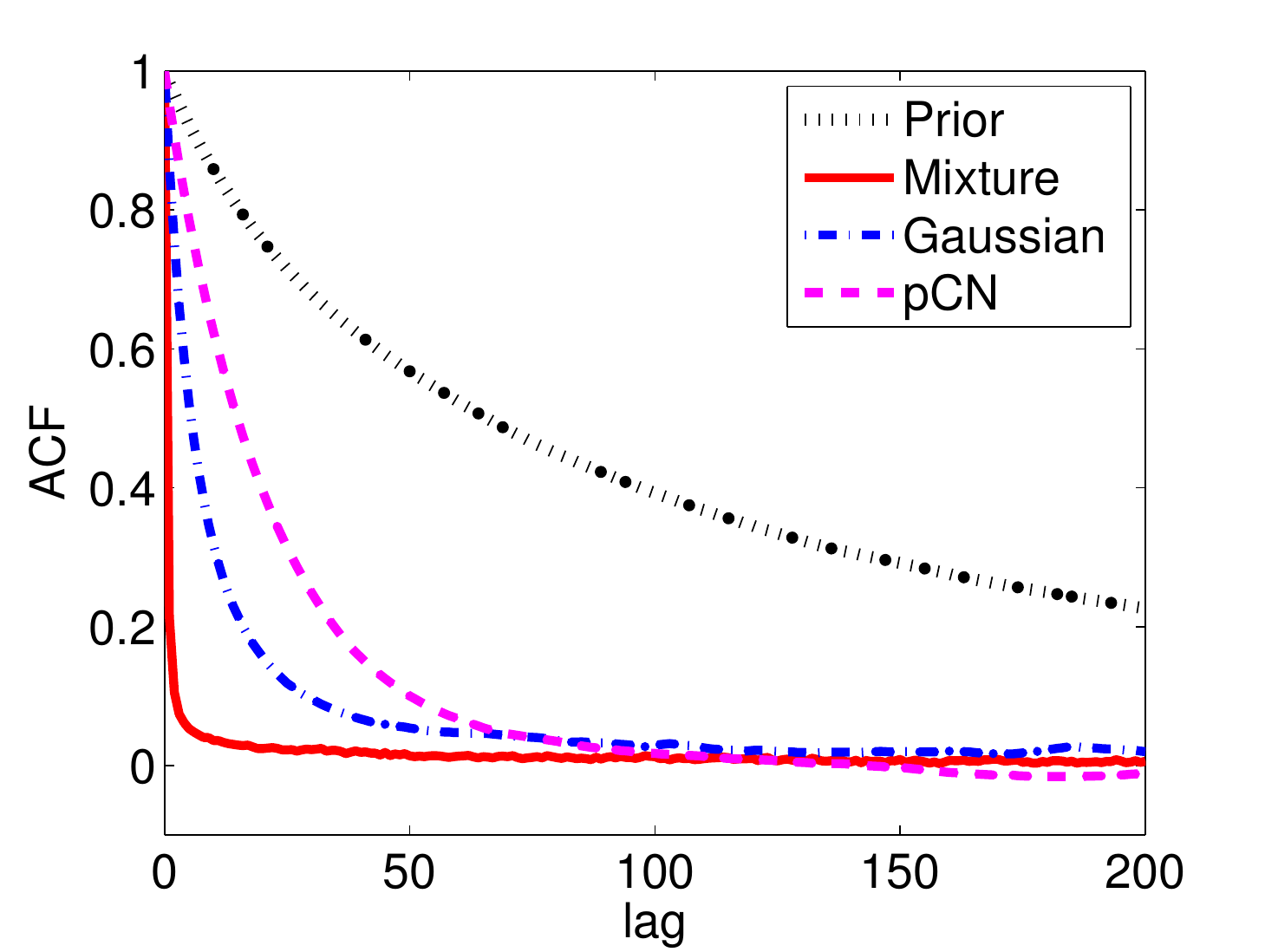}
\includegraphics[width=.5\textwidth]{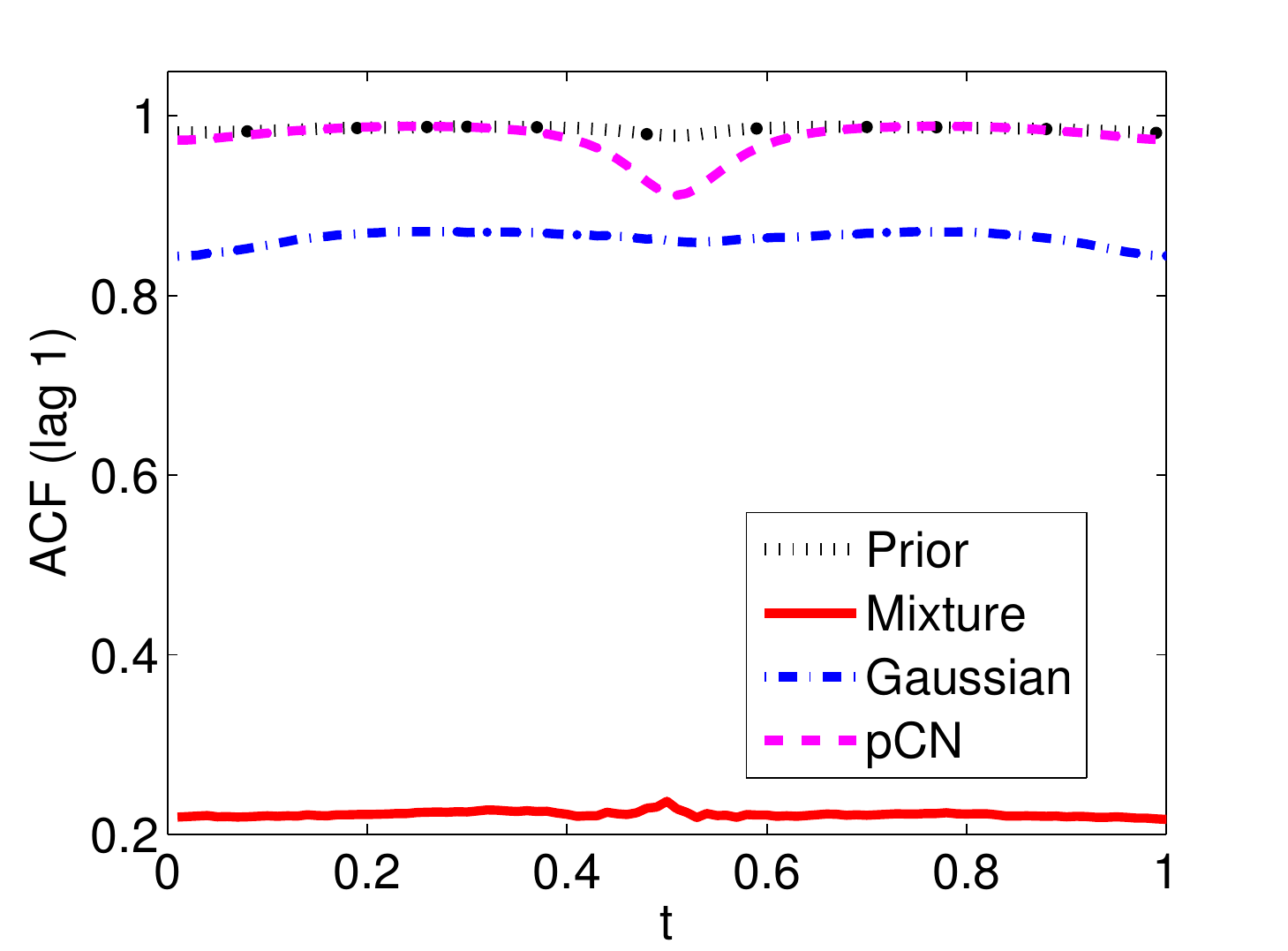}}
\caption{(for example 2) Autocorrelation functions (ACF) for the four different MCMC methods. Left: ACF of the OMF plotted as a function of lags. Right: the lag 1 ACF for 
$u$ at each grid point. }
\label{f:acf_sin}
\end{figure}

\begin{figure}
\centerline{\includegraphics[width=.5\textwidth]{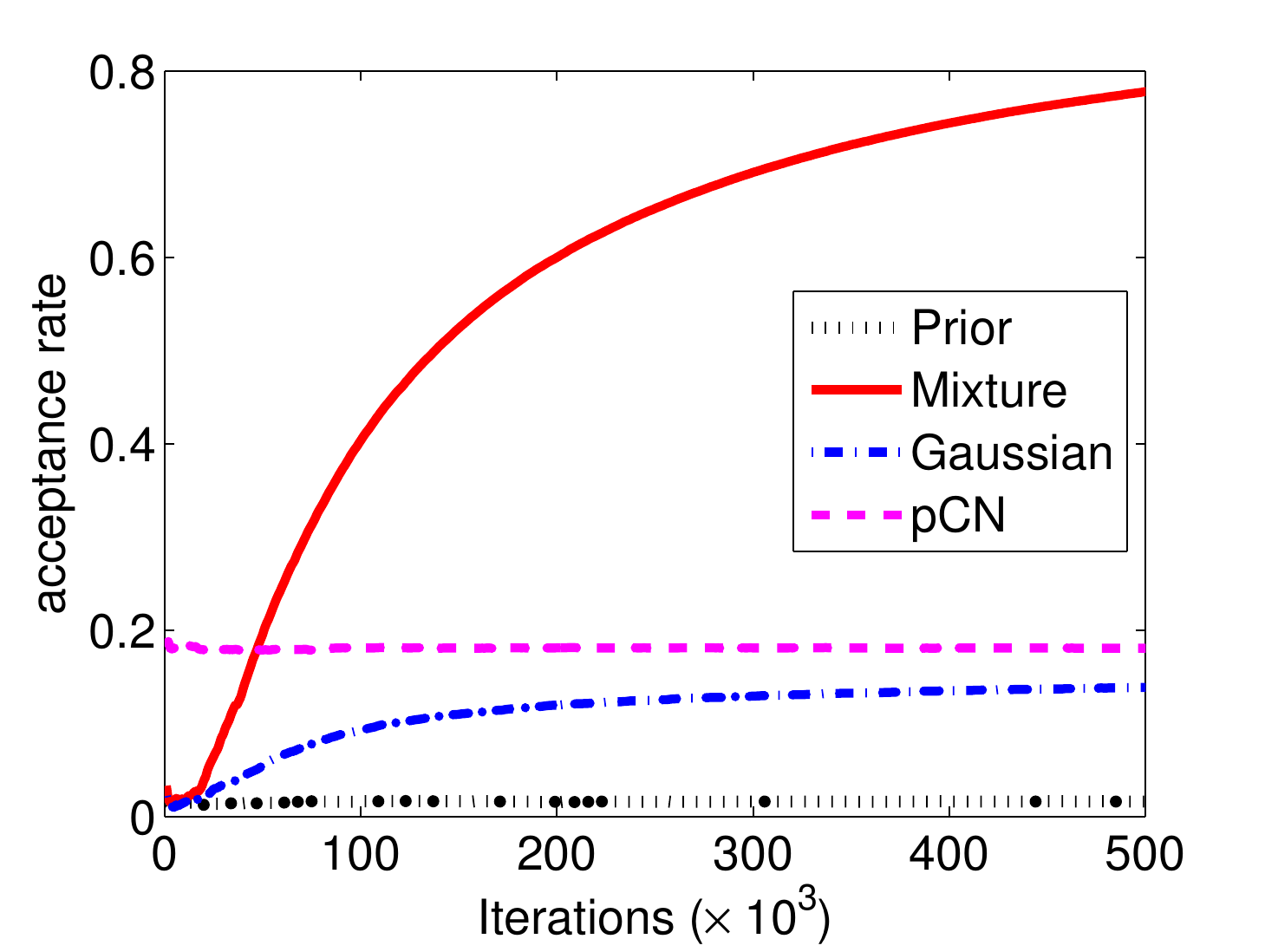}
\includegraphics[width=.5\textwidth]{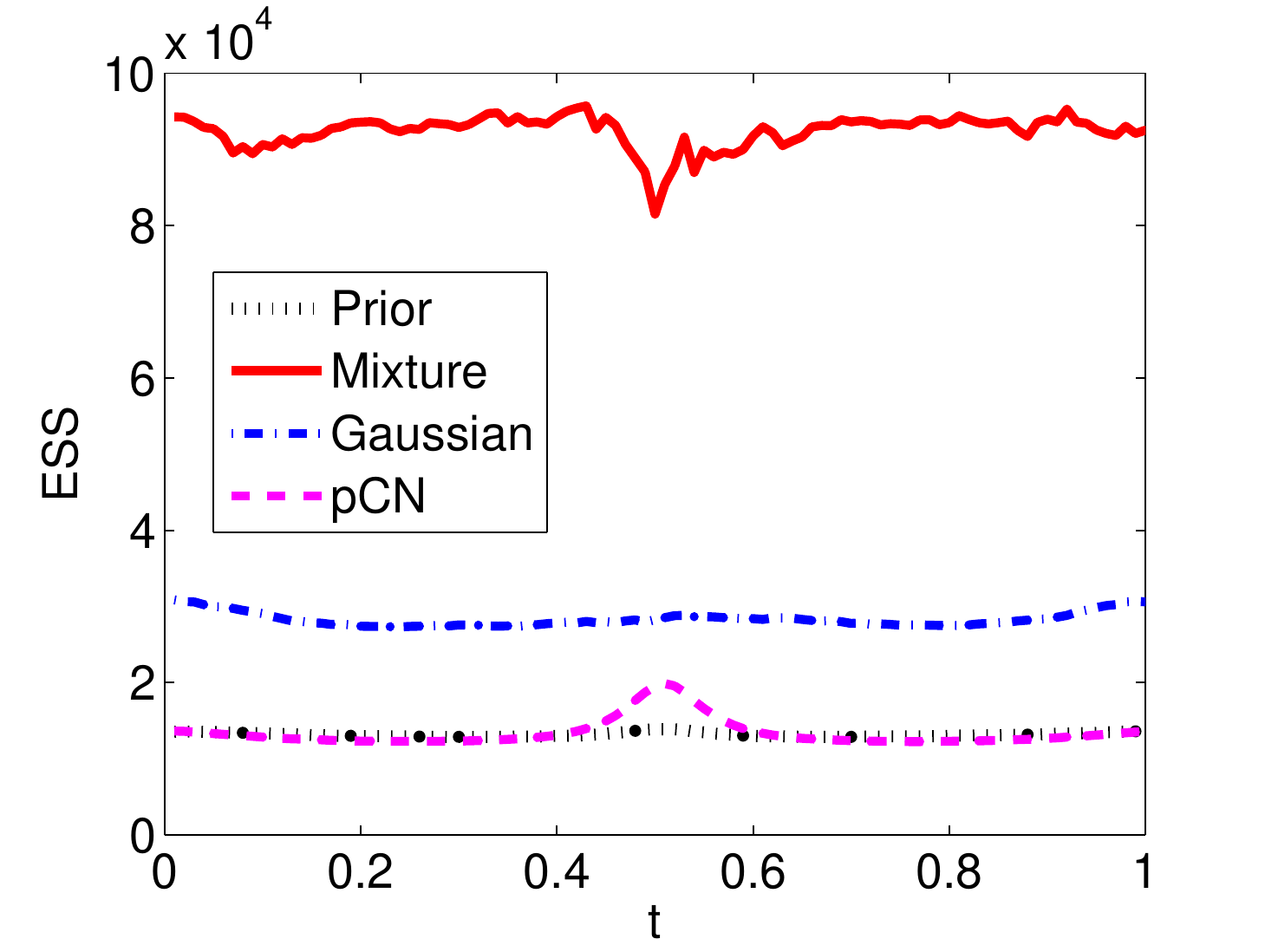}}
\caption{(for example 2) Left: the acceptance rate of the four MCMC schemes. Right: the ESS at each grid point.}
\label{f:acc-ess_sin}
\end{figure}

\begin{figure}
\centerline{\includegraphics[width=.5\textwidth]{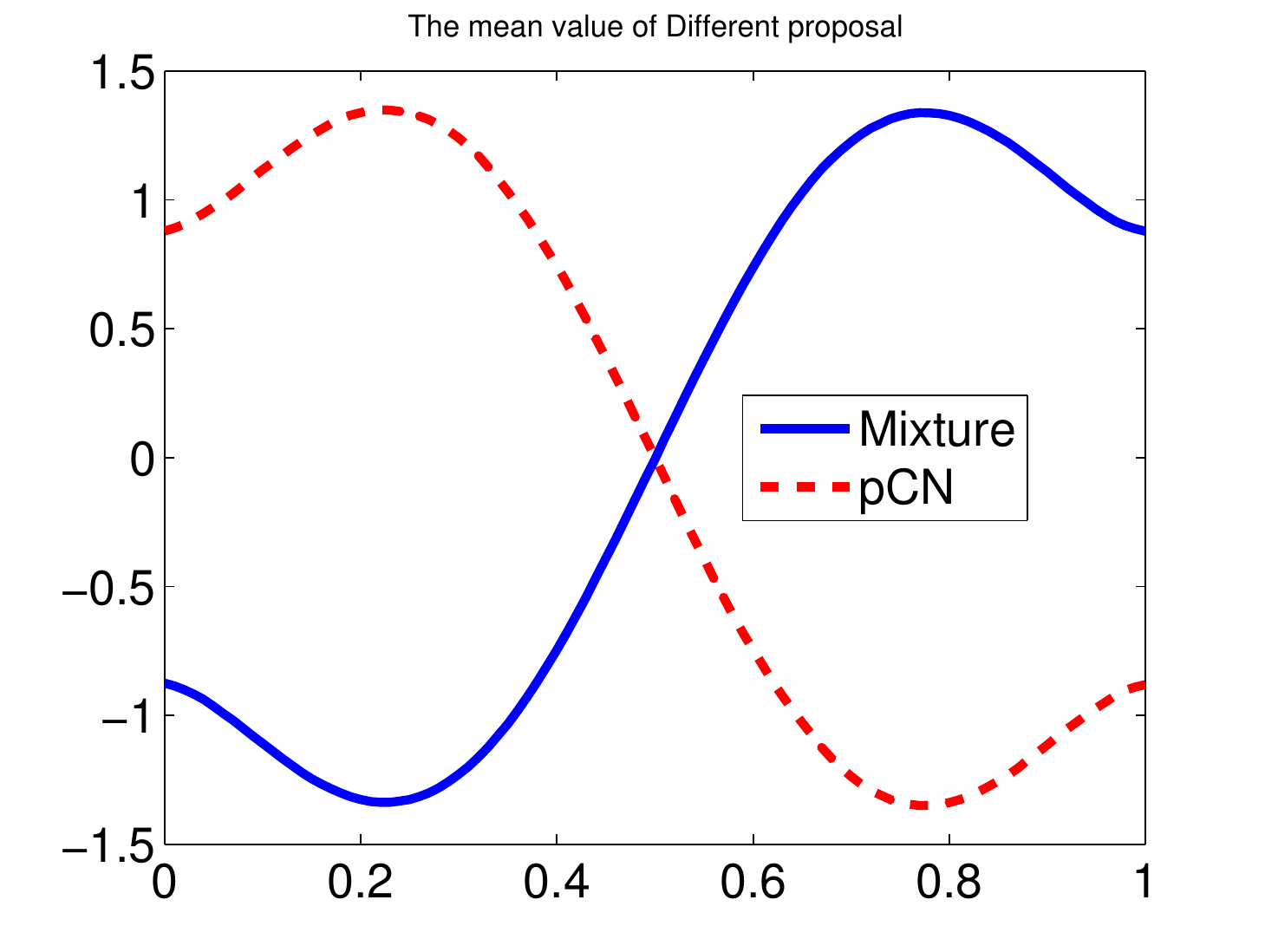}
\includegraphics[width=.5\textwidth]{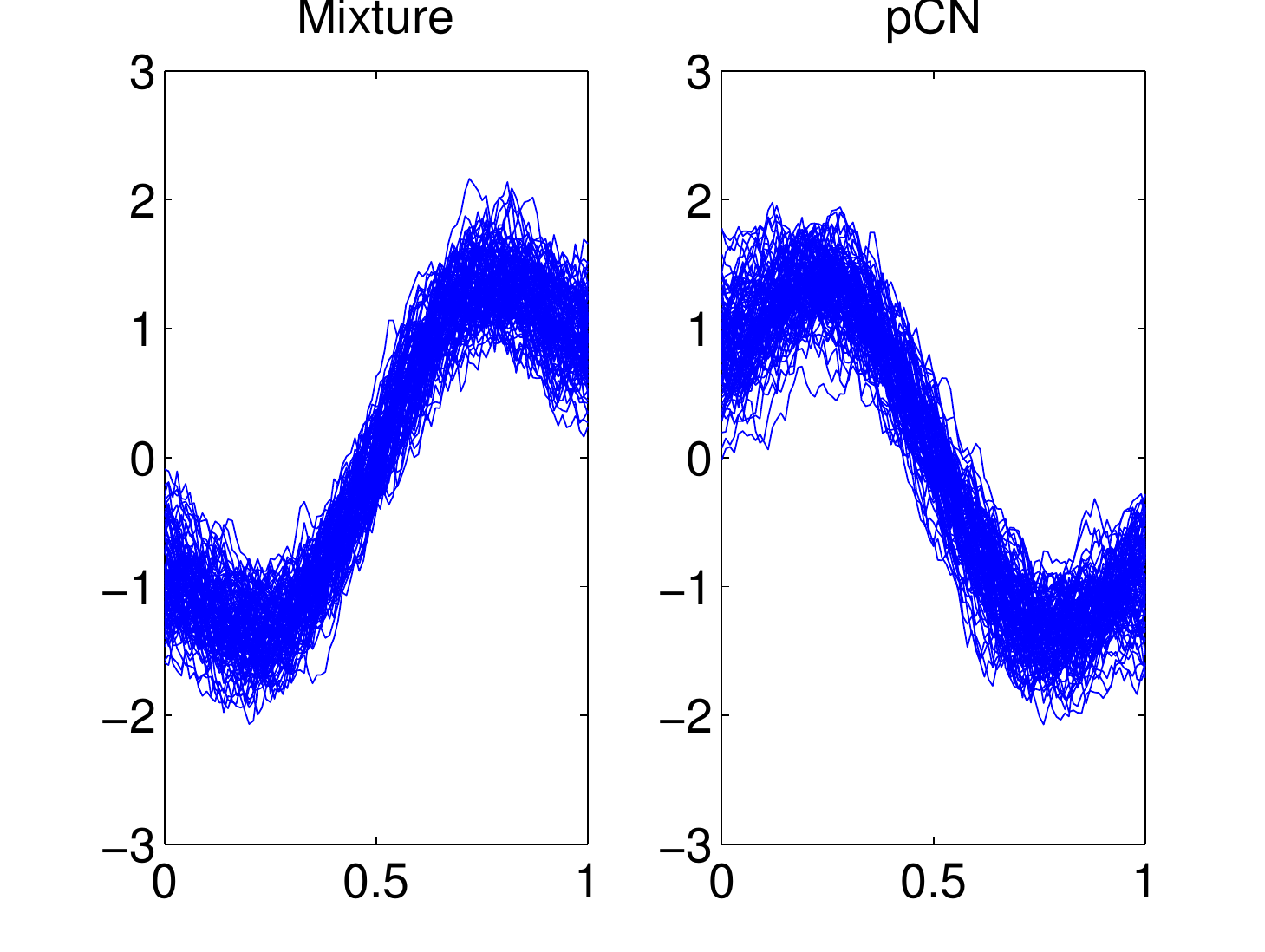}}
\caption{(for example 2) Left: the sample mean of the mixture-based IS method (solid) and that of the pCN method (dashed). Right: samples drawn by the mixture IS method and by the pCN method.}
\label{f:failed}
\end{figure}

As has been mentioned, the posterior distribution has two modes and we shall exam if the algorithms can capture both of them. 
In this respect, we randomly select 100 samples from the chain generated by each algorithm and plot them in Fig.~\ref{f:samples}.
We can see that the results of each algorithm can capture the two models of the posterior. 
Next we shall compare the efficiency of the four algorithms. As before, we first show the trace plots of the OMF for the four algorithms 
in Fig~\ref{f:trace_sin} and one can see that 
the results of the two adaptive methods and pCN all obtain fairly good mixing results, while
 the prior based IS seems to have a much slower mixing rate than the other three. 
 Fig.~\ref{f:acf_sin} (left) plots the ACF of the OMF as a function of lag and Fig.~\ref{f:acf_sin} (right) shows 
the lag 1 ACF for the unknown at each grid point.  
Both figures indicate that the adaptive IS with mixtures has the best performance in terms of ACF values. 
Fig.~\ref{f:acc-ess_sin} (left) plots the acceptance rate against the number of iterations, which shows that
the three IS algorithms perform very differently:
the prior based IS results in an acceptance rate less than $1\%$, the adaptive IS with one Gaussian results in a rate up to $17\%$,  
and  that of the adaptive IS with mixtures rises to around $80\%$ as the iteration proceeds. 
We compute the ESS of each dimension and show the results in Fig.~\ref{f:acc-ess_sin} (right), 
and we see that the ESS of the adaptive IS with mixtures is significantly higher than that of the other  three methods,
indicating that the adaptive IS with mixtures has a substantial advantage in this multimodal problem.    

Finally to understand the limitation of the proposed method, we test it on another bimodal likelihood function:
\[ \exp(-\Phi(u)) \propto \exp(-\frac12\|u-2\sin(2\pi t)\|_2^2)+\exp(-\frac12\|u+2\sin(2\pi t)\|_2^2).\]
We drew $5\times10^5$ samples with the mixture based IS algorithm and with the pCN.
We plot the mean of the samples drawn by both method in \ref{f:failed} (left),
and  in \ref{f:failed} (right), we plot 100 samples drawn by each algorithms.  
It can be seen from the figures that, both methods can only capture one mode of the posterior distribution,
indicating that, the problem becomes challenging for our method and the pCN when the  modes of the target distribution are far apart.

\subsection{Inverse heat conduction under model uncertainty}
Our last example is the inverse heat conduction~(IHC) problems, which consist of estimating temperature or heat flux density on an inaccessible boundary from a measured temperature history inside a solid.  
These problems have been studied over several decades due to their importance
in a variety of scientific and engineering applications~\cite{beck1985inverse}.
The IHC problems become 
nonlinear if the thermal properties are temperature dependent,
where the inversion is significantly more difficult than the linear ones. 
In this example we consider a one-dimensional heat conduction equation
\begin{equation}
\partialderiv{u}t=\partialderiv{}x\left[c(u)\partialderiv{u}x\right],
\label{e:heat}
\end{equation}
with initial $u(x,0)=u_o(x)$.
Here $x$ and $t$ are the spatial and temporal variable, $u(x,t)$ is the temperature, and $c(u)$ is the temperature dependent thermal conductivity,
and the length of the medium is $L$, all in dimensionless units. 
 We now assume that a heat flux is injected through the left boundary ($x=0$),  yielding a Neumann boundary condition: 
\[\partialderiv{}xu(0,t)=q(t).\]
The boundary condition~(BC) at $x=L$ is subject to uncertainty:
with probability $0.8$ it is
\begin{subequations}
 \begin{equation}
 \partialderiv{}xu(L,t) = 0,\label{e:neum}
 \end{equation}
 and with probability $0.2$ it is 
\begin{equation}
\partialderiv{}xu(L,t) = -u.\label{e:robin}
\end{equation}
\end{subequations}
The interpretation is that, the system has two possible states:  one with a perfectly insulted boundary at $x=L$, and the other
has heat diffusion at $x=L$.

Suppose that we place a temperature sensor in the medium ($x=x_\mathrm{s}$) and the goal is to infer the heat flux $q(t)$ for $t\in[0,\,T]$, from the temperature history measured by the sensor in the time interval. The schematic of this problem is shown in Fig.~\ref{f:heat}.
A similar problem without model uncertainty has been studied in \cite{li2014adaptive}. 

In the simulation, we let $L=1$, $T=2$, $c(u) = u^2+1$, $x_s=0.9$, and the initial condition be $u_o(x)=0$. 
The temperature is measured 50 times (equally spaced) and the error in each measurement is assumed to be an independent zero-mean Gaussian random variable 
with variance $0.1^2$. 
We assume the prior on $q(t)$ is a stationary zero-mean Gaussian process with a squared exponential covariance function:
\begin{equation} 
C(t,t') = \exp (-\frac{|t-t'|^2}{2d^2}), 
\end{equation}
where $d=0.3$.
The ''truth flux'' $q(t)$ is a realization of the prior (shown in Fig.~\ref{f:mean_heat}) and the data is simulated with the generated flux $q(t)$ and the boundary condition~\eqref{e:robin}.
In this problem the likelihood function becomes:
\[\frac{d\mu^y}{d\mu_0} = 0.8\exp(-\Phi_1(u))+0.2\exp(-\Phi_2(u)),\]
where $\Phi_1(u)$ corresponds to Eq.~\eqref{e:heat} with BC~\eqref{e:neum} and $\Phi_2(u)$ corresponds to Eq.~\eqref{e:heat} with BC~\eqref{e:robin}.

We draw samples from the posterior of $u(t)$ with the four MCMC schemes used in the previous examples. 
In each MCMC scheme, $1.5\times10^{5}$ draws are generated. 
In both of the adaptive IS methods, the parameters are updated after every 500 draws, and the adaptation is terminated after $10^5$ draws. To accelerate the convergence, we use tempering in the first 11 iterations (5,500 draws) with tempering parameter $\lambda =(i-1)/10 $ for $i=1...11$ .  In the RW-pCN, we choose $\beta =0.1$.

We first show the trace plot of the OMF in Fig.~\eqref{f:trace_heat}, and it is quite clear that the results of the two adaptive methods are better than those of the prior based IS and the pCN.
Because of the multimodality of the likelihood function, the posterior may have multiple modes, and to verify this, we apply the K-means method described in Section~4 to cluster the samples drawn by the four methods. 
The samples of the adaptive IS with mixtures can be successfully classified into two groups and we plotted the mean of each group in Fig.~\ref{f:mean_heat} (left), compared against the true heat flux.  
The K-means method, however, fails to separate the samples drawn by the other three methods, likely because the chains have not reached 
the target posterior distribution yet. 
We plot the means of the samples of the three methods in Fig.~\ref{f:mean_heat} (right).
Like the previous examples we show the ACF results of the four methods in Figs.~\ref{f:acf_heat},
and the acceptance rates and the ESS in Figs~\ref{f:acc-ess_heat}. In all the plots,  the adaptive IS with mixtures exhibits the best performance, followed by the IS with a single Gaussian. 

\begin{figure}
\centerline{\includegraphics[width=.7\textwidth]{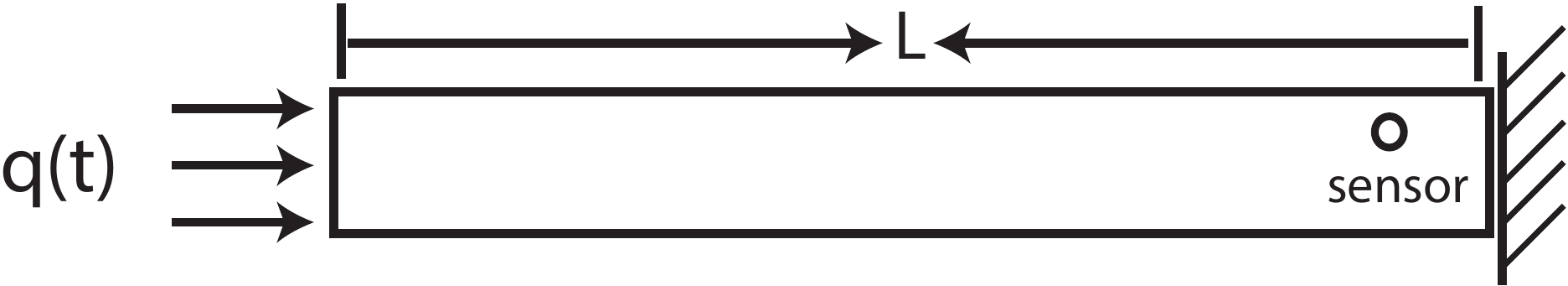}}
\caption{Schematic diagram of the IHC problem.}
\label{f:heat}
\end{figure}

\begin{figure}
\centerline{\includegraphics[width=.9\textwidth]{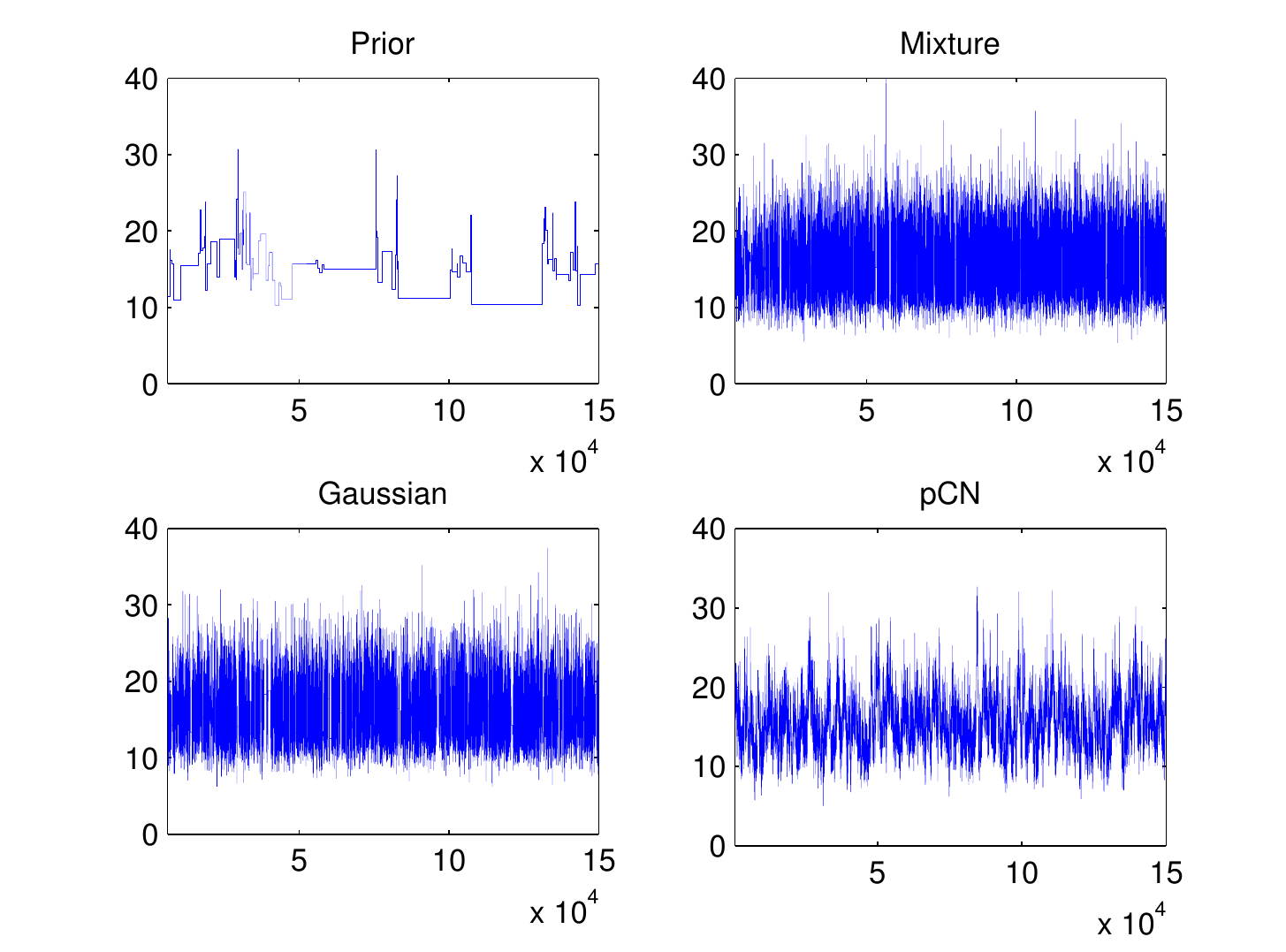}}
\caption{(for example 3) The trace plots of the OMF for the four different MCMC schemes.}
\label{f:trace_heat}
\end{figure}

\begin{figure}
\centerline{\includegraphics[width=.5\textwidth]{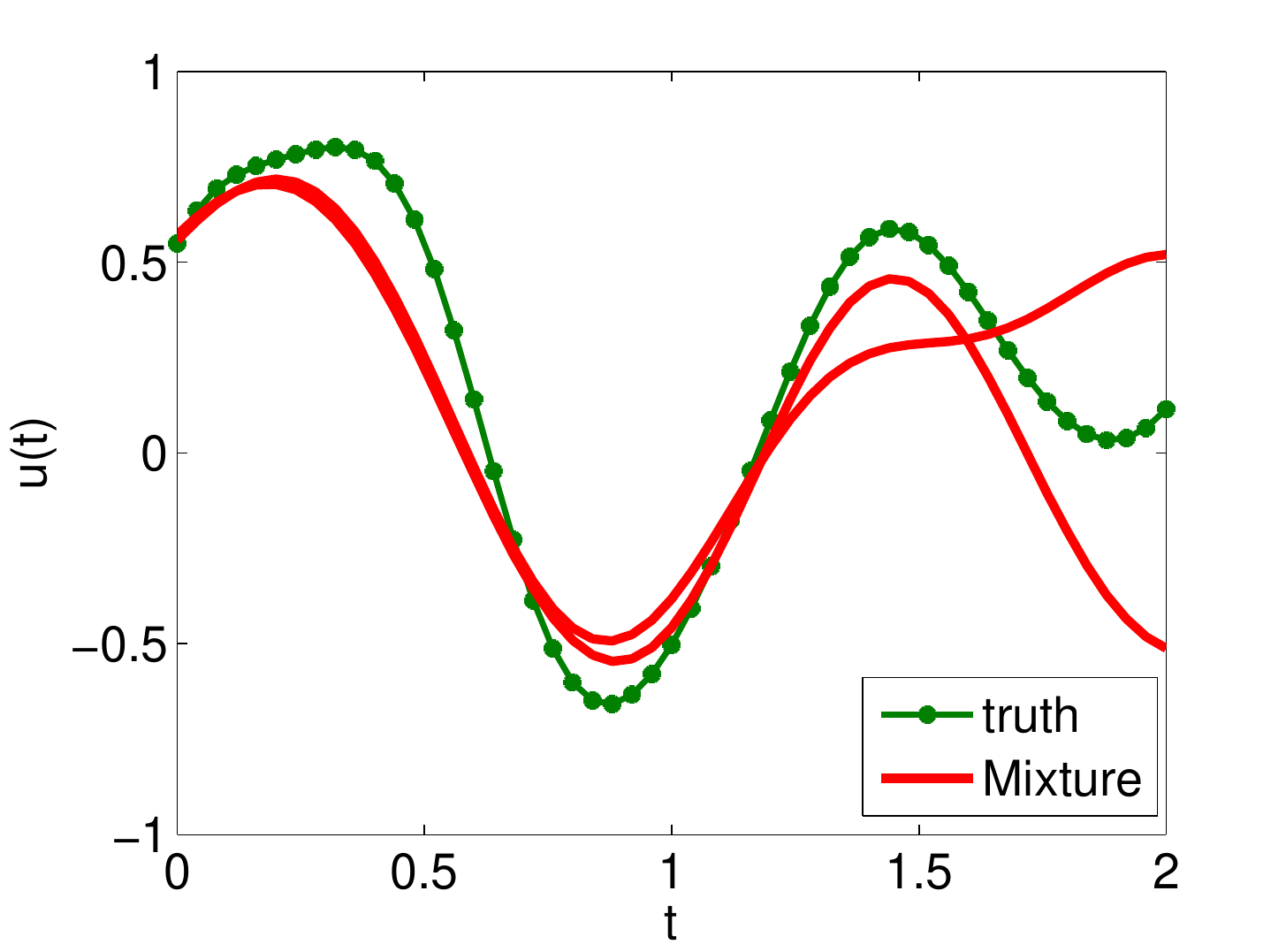}
\includegraphics[width=.5\textwidth]{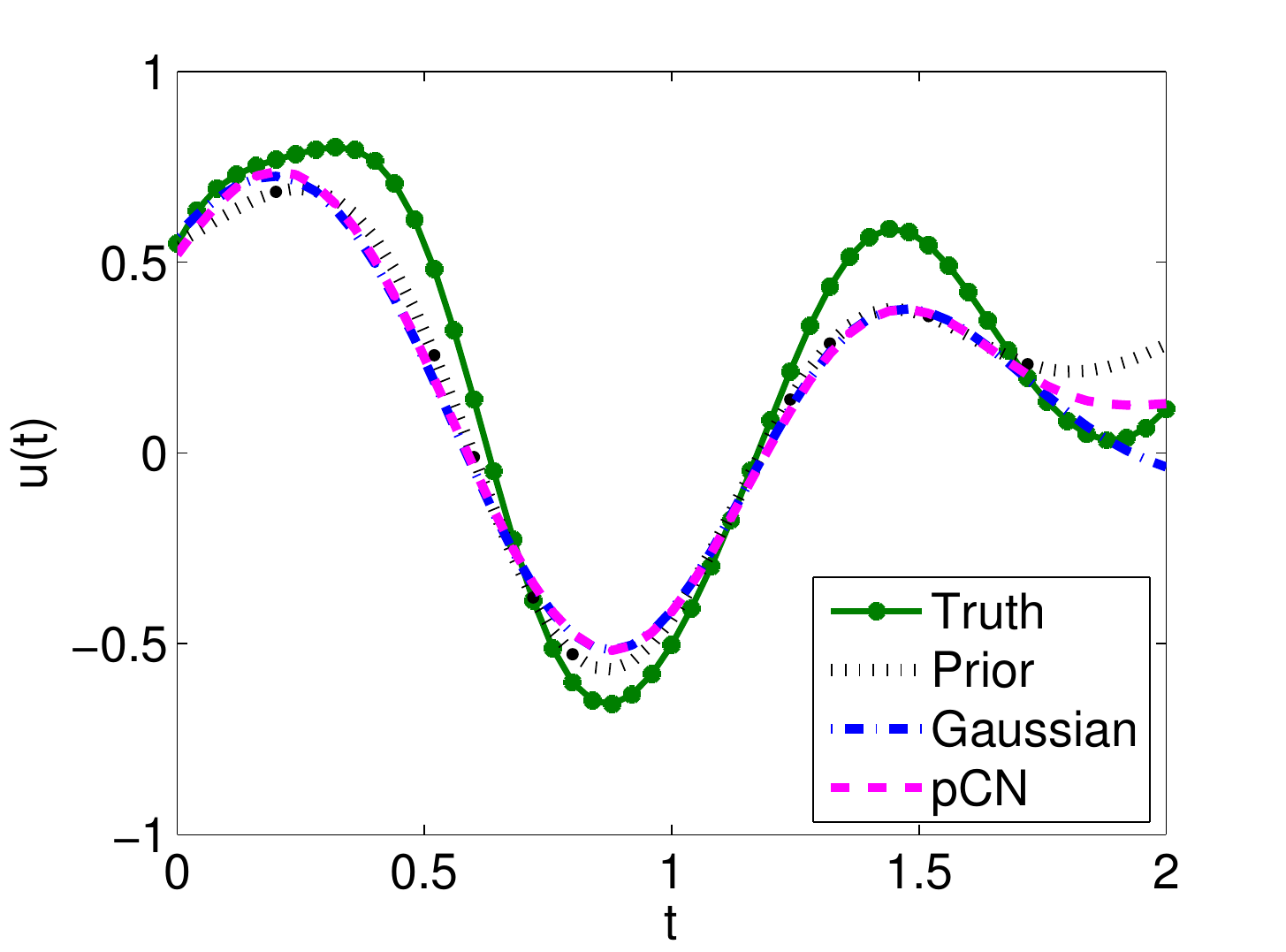}}
\caption{(for example 3)  Left: the means of the samples in each cluster of the chain drawn by the IS with mixtures and the true flux.
Right: the means of the samples drawn by the adaptive IS with a single Gaussian, the prior based IS and the RW-pCN.   }
\label{f:mean_heat}
\end{figure}

\begin{figure}
\centerline{\includegraphics[width=.5\textwidth]{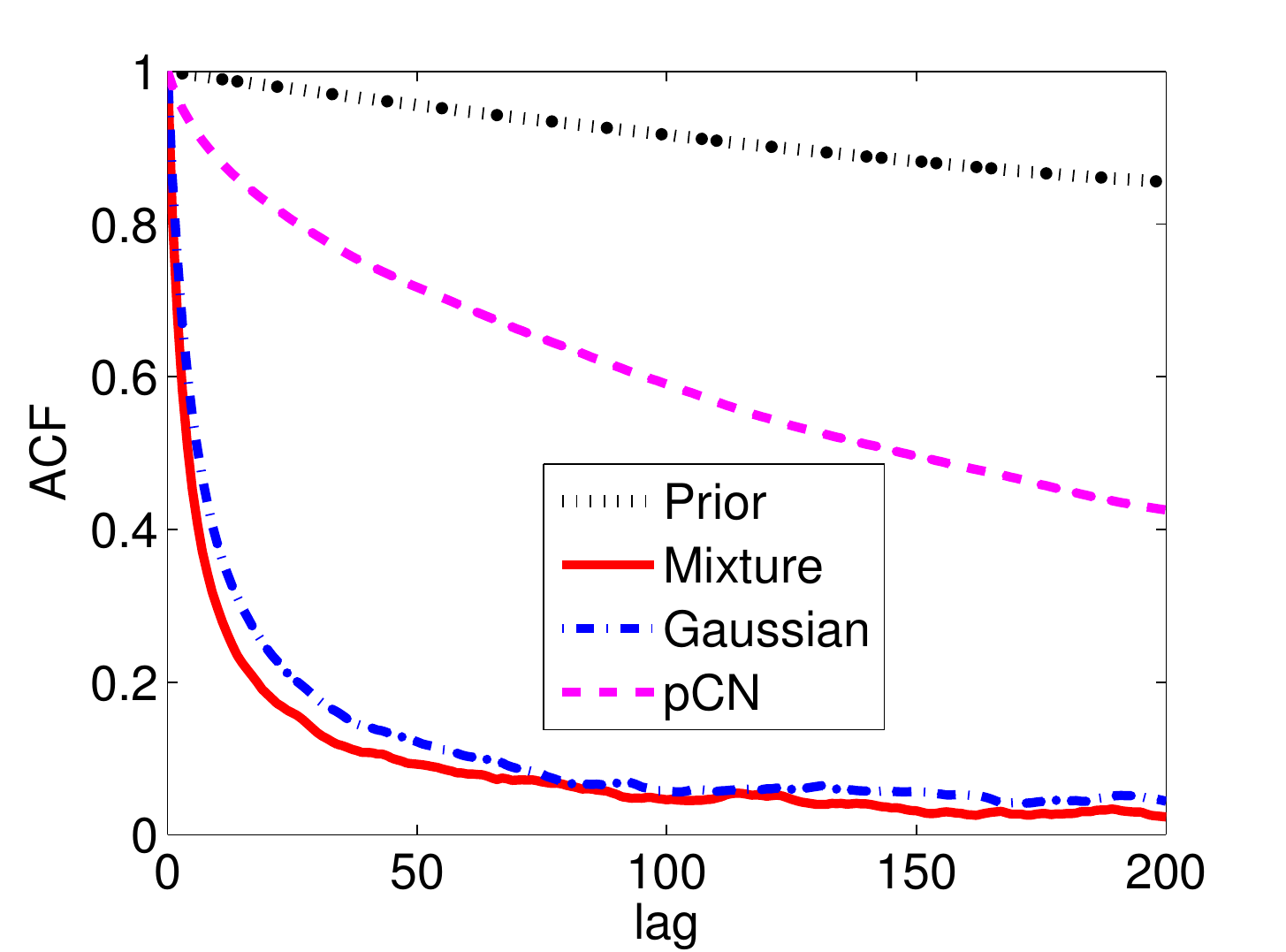}
\includegraphics[width=.5\textwidth]{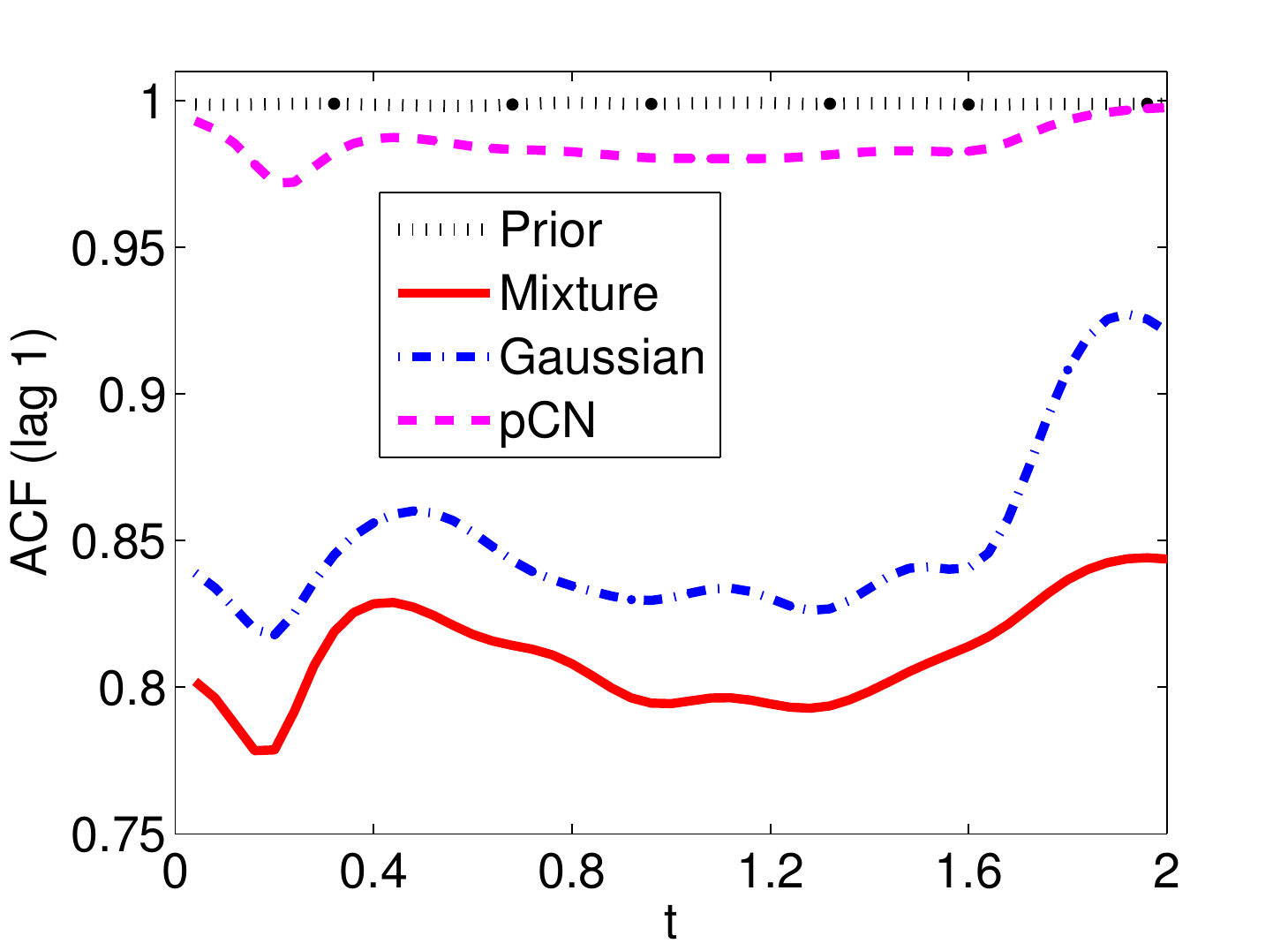}}
\caption{(for example 3) Autocorrelation functions (ACF) for the four different MCMC methods. Left: ACF of the OMF plotted as a function of lags. Right: the lag 1 ACF for 
$u$ at each grid point. }
\label{f:acf_heat}
\end{figure}

\begin{figure}
\centerline{\includegraphics[width=.5\textwidth]{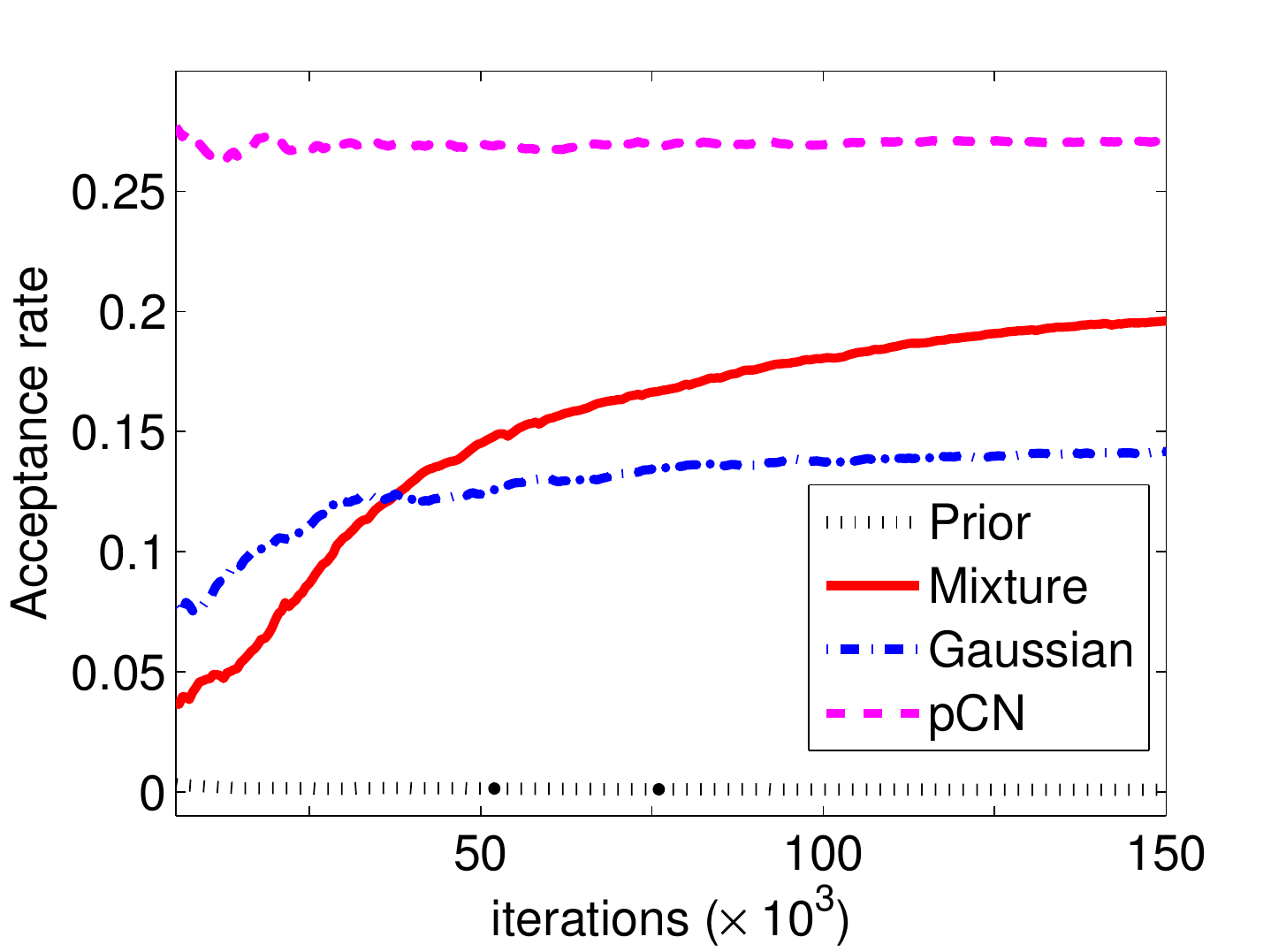}
\includegraphics[width=.5\textwidth]{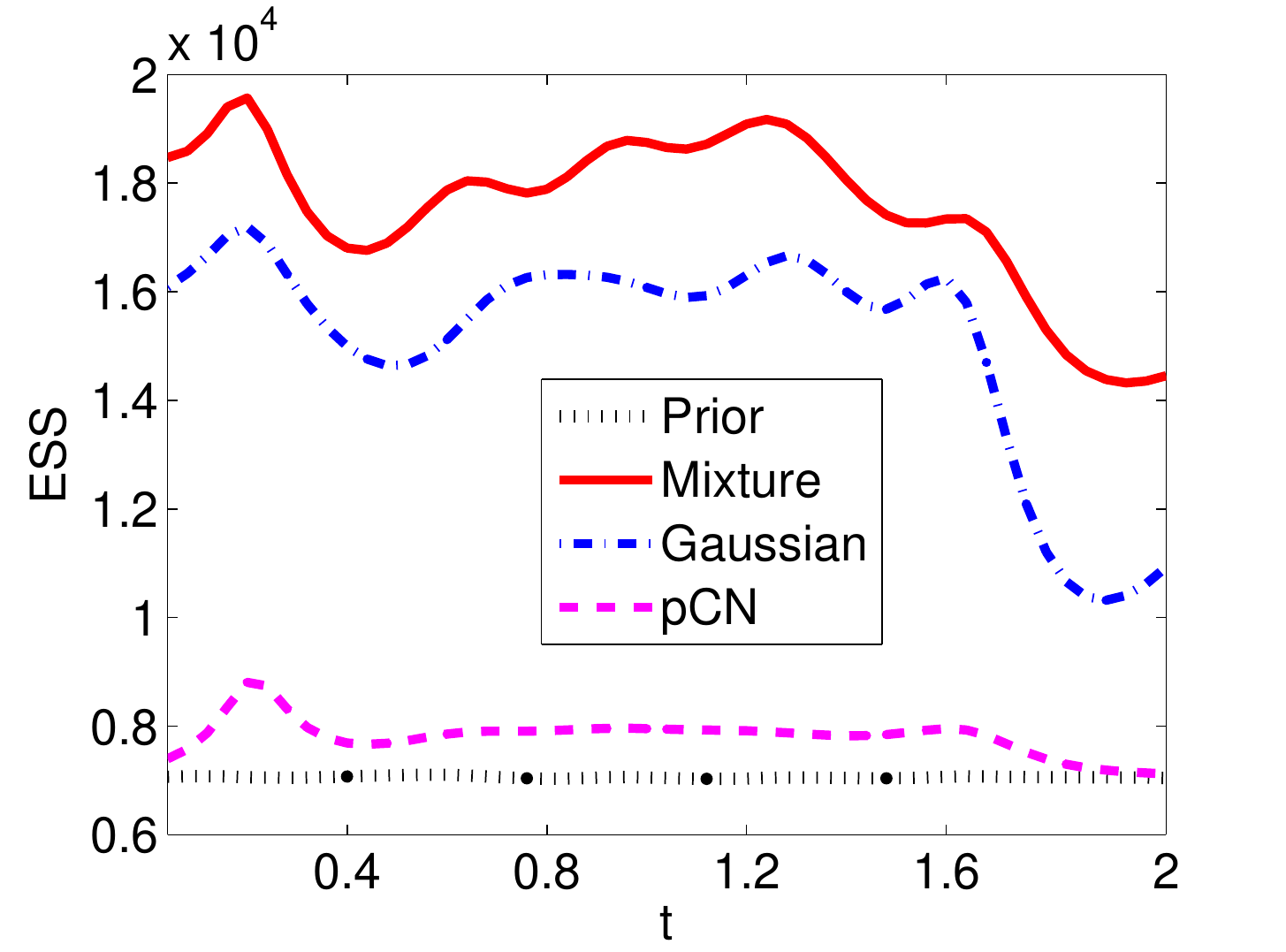}}
\caption{(for example 3) Left: the acceptance rate of the four MCMC schemes. Right: the ESS at each grid point.}\label{f:acc-ess_heat}
\end{figure}

\section{Conclusions}
In conclusion, we have presented an adaptive IS algorithm for infinite dimensional Bayesian inference. 
Namely we choose a Gaussian mixture with a particular parametrization as our proposal, and adaptively adjust the parameter values using sample history. 
We prove that the proposed algorithm is well-defined in function space and thus is dimension independent. 
We also develop an efficient algorithm based on clustering to compute the parameter values in each iteration.
We demonstrate the efficiency of the proposed method with numerical examples and in particular we show that it 
performs well for multimodal posteriors. We emphasize that, the proposed method is easy to implement, treating the problem as a black box model, 
and requiring no information on the mathematical structure of the forward model.

As has been demonstrated by the numerical examples, the mixture proposals can generally provide faster mixing rates  
than the single Gaussian, thanks to its higher flexibility. 
On the other hand, given that the Gaussian approximation is less complex computationally (without the clustering step), we recommend to use the single Gaussian approximation in problems where the posterior distributions do not deviate too much from a Gaussian measure,
and to use mixtures for strongly non-Gaussian posteriors. 

There are number of possible extensions of the work. First in this work we approximate the solution to the KLD minimization problem 
with clustering. It is possible that, if we can modify the standard EM algorithm and use it to solve the optimization problem directly, 
we may obtain better a mixture proposal in each iteration and improve the sampling efficiency. 
Secondly, the intrinsic dimensionality $K$ is of essential importance for our method, and
in the present work,  $K$ is determined rather heuristically.
Thus developments of more effective and theoretically justified methods certainly deserve further studies. 
Finally, the algorithm developed here is based on independence sampler, and we 
are also interested in extending the ideas to the development of adaptive infinite dimensional random walk algorithms.    
We plan to investigate these problems in the future.

\bibliographystyle{siam}
\bibliography{mcmc}

\end{document}